 \documentclass[journal]{IEEEtran}
\ifCLASSINFOpdf
\else
\fi
\usepackage[pdftex]{graphicx}

\usepackage{enumerate}
\usepackage[caption=false,font=footnotesize]{subfig}
\usepackage{color}

\usepackage{comment}
\excludecomment{comment}
\specialcomment{notes}{\begingroup\sffamily\color{blue}}{\endgroup}
\newcommand{\FA}{\mathbf{A}}   % Function A
\newcommand{\FB}{\mathbf{B}}   % Function B
\newcommand{\FH}{\mathbf{H}}   % Function H
\newcommand{\Fu}{\mathbf{u}}	   % Function u
\newcommand{\Fx}{\mathbf{x}}	   % Function x
\newcommand{\Fv}{\mathbf{v}}	   % Function v
\newcommand{\FI}{\mathbf{I}}	   % Function I
\newcommand{\FW}{\mathbf{W}}	   % Function W
\newcommand{\FV}{\mathbf{V}}
\newcommand{\FU}{\mathbf{U}}	   % Function V

\newcommand{\FN}{\mathbf{N}}	   % BF N
\newcommand{\Ft}{\mathbf{t}}	   % BF t
\newcommand{\FT}{\mathbf{T}}	   % BF T
\newcommand{\Fs}{\mathbf{s}}	   % BF s
\newcommand{\Ff}{\mathbf{f}}	   % BF f

\newcommand{\SA}{\mathbf{A^{[N]}}}
\newcommand{\SB}{\mathbf{B^{[N]}}}

\newcommand{\SC}{\mathbf{C^{[N]}}}

\newcommand{\SxZ}{\mathbf{x^{[N]}_0}}
\newcommand{\Sx}{\mathbf{x^{[N]}}}
\newcommand{\Sxt}{\mathbf{x^{[N]}_t}}

\newcommand{\SxT}{\mathbf{x^{[N]}_T}}

\newcommand{\DSxt}{\mathbf{\dot{x}^{[N]}_t}}

\newcommand{\Su}{\mathbf{u^{[N]}}}
\newcommand{\Sut}{\mathbf{u^{[N]}_t}}
\newcommand{\Sutau}{\mathbf{u^{[N]}_\tau}}
\newcommand{\BA}{\mathbb{A}}
\newcommand{\BB}{\mathbb{B}}
\newcommand{\BC}{\mathbb{C}}
\newcommand{\BS}{\mathbb{S}}
\newcommand{\BR}{\mathds{R}}
\newcommand{\BN}{\mathds{N}}
\newcommand{\BT}{\mathbb{T}}
\newcommand{\BW}{\mathbb{W}}   
\newcommand{\BSA}{\mathbb{A^{[N]}}}
\newcommand{\BSB}{\mathbb{B^{[N]}}}
\newcommand{\BSC}{\mathbb{C^{[N]}}}
\newcommand{\BP}{\mathbb{P}}
\newcommand{\BSP}{\mathbb{P^{[N]}}}   %  step funciton
\newcommand{\BSPt}[1]{\mathbb{P}^{\mathbb{[N]}}_{#1}}
\newcommand{\BPZ}{\mathbb{P}_0}
\newcommand{\BSPZ}{\mathbb{P}^{[\mathbb{N}]}_{0}}
\newcommand{\BAP}{\widetilde{\mathbb{P}}^{\mathbb{N}}} 
\newcommand{\BAPt}[1]{{\widetilde{\mathbb{P}}}^{\mathbb{N}}_{#1}}

\newcommand{\BI}{\mathbb{I}}     % Function I

\newcommand{\GS}{\tilde{\mathcal{W}}}
\newcommand{\GSZ}{\tilde{ \mathcal{W}}_0}
\newcommand{\GSO}{\tilde{\mathcal{W}}_1}

\newcommand{\ES}{\mathcal{W}} % Exact Space functions without taking cut distance zero.
\newcommand{\ESZ}{\mathcal{W}_0}
\newcommand{\ESO}{\mathcal{W}_1}

\newcommand*\TRANS{{\mathpalette\doTRANS\empty}}
\makeatletter
\newcommand*\doTRANS[2]{\raisebox{\depth}{$\m@th#1\intercal$}}
\makeatother

\usepackage[noadjust,sort,compress]{cite} 
\usepackage{amssymb,amsmath,dsfont}
\usepackage[T1]{fontenc}
\usepackage[standard,amsmath,thmmarks]{ntheorem}
\newtheorem{assumption}{Assumption}

\begin{document}

% \title{Graphon-Network Regulation of Networks of Linear Systems}
\title{Graphon Control of Large-scale Networks of Linear Systems}


\author{Shuang~Gao,~\IEEEmembership{Member,~IEEE,}
        and~Peter E.~Caines,~\IEEEmembership{Life~Fellow,~IEEE}% <-this % stops a space
\thanks{*Supported in part by NSERC (Canada), and the U.S. ARL and ARO grant W911NF1910110.}
\thanks{The authors are with the Department of Electrical and Computer Engineering, McGill University,
  Montreal, QC, Canada. (email: {\tt\small    $\{$sgao,peterc$\}$@cim.mcgill.ca})}% <-this % stops a space
% 
 %\thanks{Manuscript received July 13, 2018; revised May 29, 2019. }
}

\maketitle


%%%%%%%%%%%%%%%%%%%%%%%%%%%%%%%%%%%%%%%%%%%%%%%%%%%%%%%%%%%%%%%%%%%%%%%%%%%%%%%%
\begin{abstract}
To achieve control objectives for extremely large-scale complex networks using standard methods is essentially intractable.
In this work a theory of the approximate control of complex network systems is proposed and developed by the use of graphon theory and the theory of infinite dimensional systems. 
First, graphon dynamical system models are formulated in an appropriate infinite dimensional space in order to represent arbitrary-size networks of linear dynamical systems, and to define the convergence of sequences of network systems with limits in the space. Exact controllability and approximate controllability of graphon dynamical systems are  then investigated.
Second, the minimum energy state-to-state control problem and the linear quadratic regulator problem for systems on complex networks are considered. The control problem for graphon limit systems is solved in each case and approximations are defined which yield control laws for the original control problems. Furthermore, convergence properties of the approximation schemes are established. 
A systematic control design methodology is developed within this framework. 
Finally, numerical examples of networks with randomly sampled weightings are presented to illustrate the effectiveness of the graphon control methodology.
\end{abstract}

%

% \fi
%x
\IEEEpeerreviewmaketitle

\section{Introduction}
Complex network systems  such as the Internet of Things (IoT), electric, neuronal, food web, epidemic, stock market and social networks,
are ubiquitous,  and they have been the focus of much research over the past 20 years. In particular, researchers have been studying networks of interacting dynamical systems  to learn which collective behaviours may  emerge from system interactions over complex networks (\cite{olfati2006flocking,leonard2007collective,ogren2004cooperative}). 
Furthermore, in addition to the structural properties of networks,  system theoretic notions such as controllability, observability, consensus dynamics and  synchronization have been widely applied to systems on networks ({\cite{NodalDynamics2012,liu2011controllability, arenas2008synchronization, wang2002pinning,yan2015spectrum,pasqualetti2014controllability,YouXie2011network,wang2002synchronization}}). 
However, to
achieve general control objectives for  extremely large scale networks with complex interconnections (henceforth, complex networks) using these standard methods is essentially an intractable task. 

 Graphon theory, introduced and developed in recent years by
 L. Lov{\'a}sz, B. Szegedy, C. Borgs, J. T. Chayes, V. T. S{\'o}s, and K. Vesztergombi among others (see \cite{lovasz2006limits, borgs2008convergent, borgs2012convergent, lovasz2012large}), provides a theoretical tool to characterize complex graphs and graph limits. This work draws on graph theory, measure theory, probability, and functional analysis, and has been applied in different areas such as games \cite{parise2018graphon,PeterMinyiCDC18GMFG}, signal processing \cite{morency2017signal},  network centrality \cite{avella2018centrality}, and the heat equation \cite{medvedev2014nonlinear}. 

We propose a graphon based control methodology for controlling complex network systems. 
The general graphon control strategy consists of the following steps:
\begin{enumerate}[1)]
	\item  Identify the graphon limit of the sequence $\tilde{S}$ of networks as the number of nodes goes to infinity. 
 \item  Solve the corresponding control problem for the limit graphon dynamical system.  
 \item Approximate the control law for the limit system  so as to generate control laws for the application to any given finite system along the sequence $\tilde{S}$ of  finite network systems. 
\end{enumerate}
Specifically, in this paper, the minimum energy state-to-state control problem and the linear quadratic regulator problem are solved for complex network systems using this graphon control strategy.  

The main contributions of this paper include: 
\begin{enumerate}[1)]
\item the formulation of graphon differential equations and graphon dynamical control systems, which allows us to represent linear control systems on arbitrary size networks and compare systems of different sizes. This further permits us to design the graphon control methodology based on the network limit.
\item the  graphon state-to-state control methodology to solve state-to-state control problem on complex networks.
\item the proposed graphon linear quadratic regulation methodology to solve linear quadratic regulator problems on large-scale networks.
\end{enumerate}

This paper contains the complete proofs omitted in the previous articles \cite{ShuangPeterCDC17, ShuangPeterMTNS18a,ShuangPeterCDC18} and the extension of previous results, as well as new numerical examples. 

The paper is organized as follows: In Section II, the fundamentals of graphon theory are presented, followed by the development of the graphon unitary operator algebra and graphon differential equations. Section III introduces the network system model and its equivalent representation by the graphon dynamical system. Section IV presents the properties of graphon dynamical systems, including existence and uniqueness of the solution and controllability. In Section V and Section VI, the graphon control strategies for the state-to-state control problem and the linear quadratic regulator problem are presented respectively. For each problem, the approximation method is developed and the corresponding convergence properties are established.   Section VII contains numerical examples to illustrate the graphon control methodology.  

\emph{Notation:} Bold face letters (e.g. $\FA$, $\FB$, $\Fu$) are used to represent graphons and functions.  Blackboard bold letters (e.g. $\BA$, $\BB$) are used to denote linear operators which are not necessarily compact. Let $\BI$ denote the identity operator.  Let $\langle\cdot ,  \cdot \rangle$ denote inner product for $L^2[0,1]$ and $\|\cdot \|$ represent norm.  $\mathds{1}_S(\cdot)$ denotes the indicator function for a set $S$, that is, $\mathds{1}_S(x)=1$ if $x\in S$, and $0$ otherwise. $\mathbf{1}_{S}$ denotes the $L^2[0,1]$ function with $1$ in $S\subset[0,1]$ and $0$ in $[0,1]\backslash S$. 
The set of all real numbers and that of all natural numbers (excluding $0$) are respectively denoted by $\BR$ and $\BN$. 
\section{Preliminaries}
 \subsection{Graphs, Adjacency Matrices and Pixel Pictures}
The underlying structure of a network can be described by a graph $G=(V, E)$ specified by a node set $V$  and an edge set $E$ which represents the connections between nodes.  An equivalent representation of a graph $G=(V, E)$ by a matrix called an \emph{adjacency matrix} is defined to be the square $|V|\times |V|$ matrix $A$ such that an element $A_{ij}$ is one when there is an edge from node $i$ to node $j$, and zero otherwise. If the graph is a weighted graph where edges are associated with weights, then the adjacency matrix has corresponding weighted elements.

Another representation of the adjacency matrix is given by a pixel diagram where the 0s are replaced by white squares and the 1s  by black squares.
The whole pixel diagram is presented in a unit square, so the square elements have sides of length $\frac1n$, where $n$ is the number of nodes.

\begin{figure}[h]
\centering
	\includegraphics[height=1.8cm]{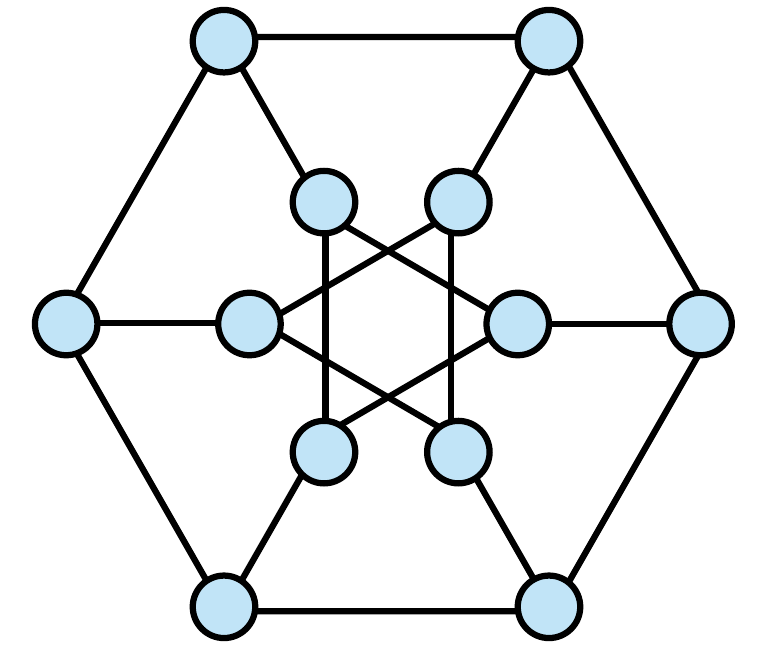} \quad
	\includegraphics[height=1.8cm]{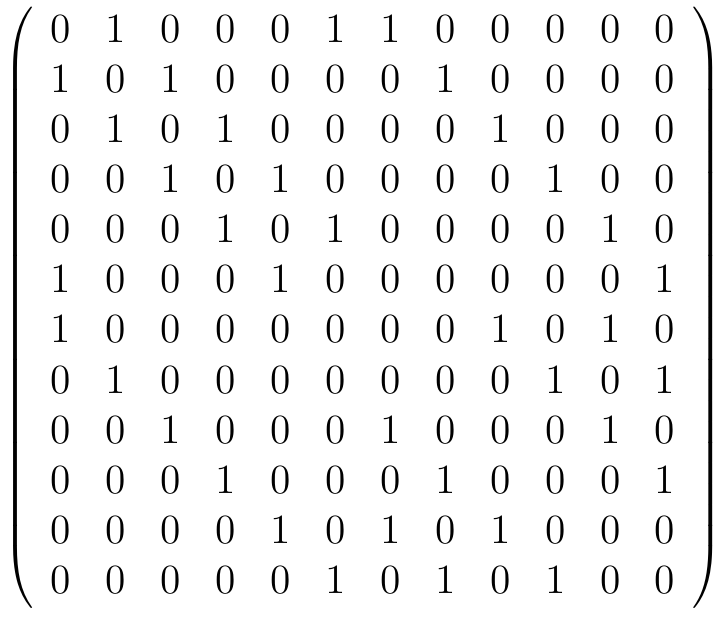} \quad
	\includegraphics[height=1.8cm]{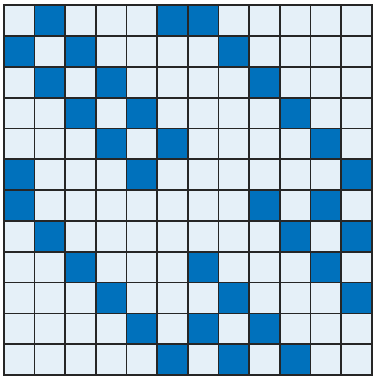}
	\caption{{D\"urer Graph, adjacency matrix, pixel diagram }}
	\label{fig: graph-adjacencymatrix-pixelpicture}
\end{figure}
\subsection{Graphons}

A meaningful convergence with respect to the \emph{cut metric} is defined for sequences of dense and finite graphs \cite{lovasz2012large}.  Graphons are then the limit objects of converging graph sequences. This concept is illustrated by a sequence of half graphs \cite{lovasz2012large} represented by
a sequence of pixel diagrams on the unit square converging to its limit in Fig. \ref{fig: converge-in-pixel-pictures}. Readers are referred to \cite{lovasz2012large} for more examples of convergent graph sequences such as uniform attachment graphs, complete bipartite graphs, and Erd\"os-R\'enyi graphs.
Exchangeable random graphs can also be modeled by graphons \cite{orbanz2014bayesian}.

 \begin{figure}[h]
 	\centering
	\includegraphics[height=1.8cm]{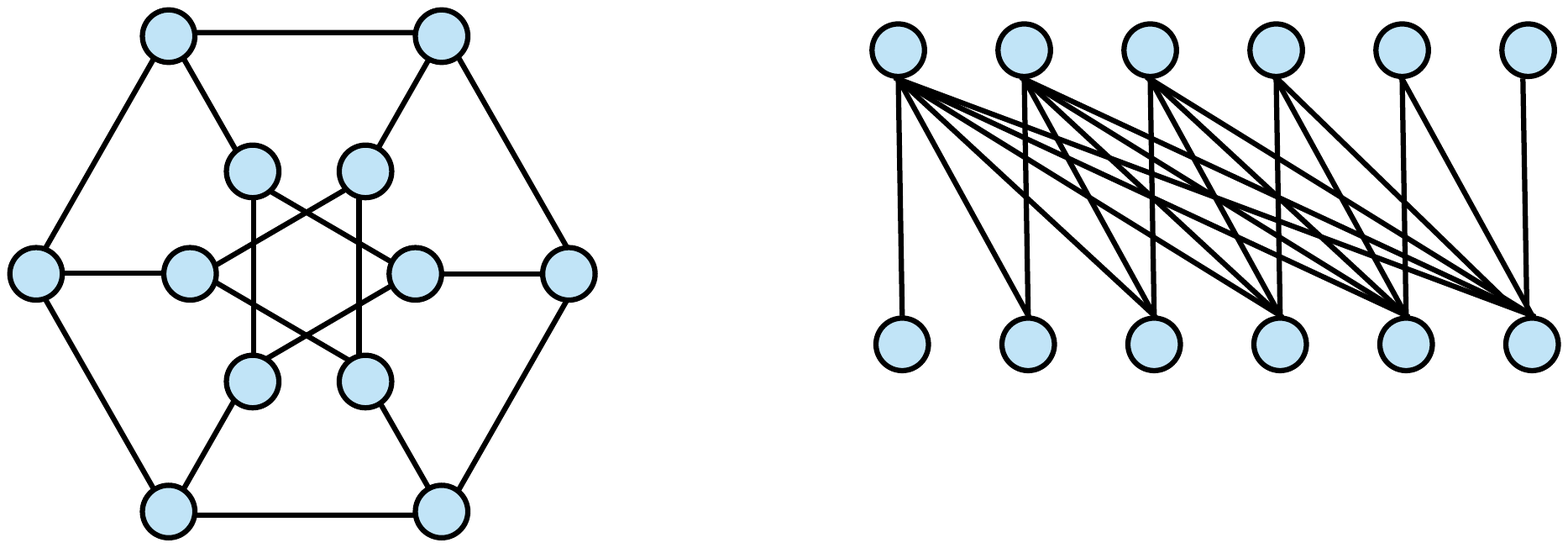}\qquad
	\includegraphics[height=1.8cm]{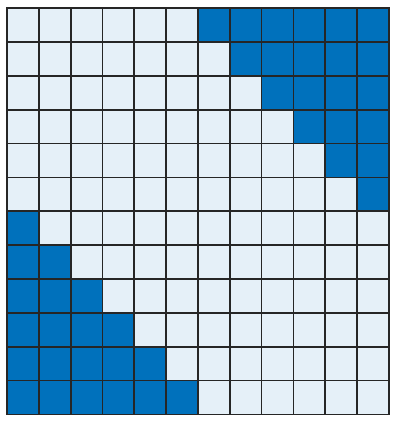}\qquad
	\includegraphics[height =1.8cm]{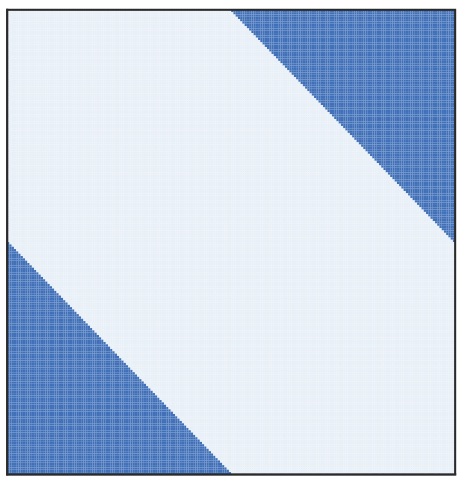}
	\caption{{Graph sequence converging to its limit}}
	\label{fig: converge-in-pixel-pictures}
\end{figure}
%  \begin{figure}[h]
%   \centering
%   \includegraphics[height =1.8cm]{fig/E-R_graph}
%   \caption{{Erd\"os–R\'enyi graph sequence converging to its limit}}
%   \label{fig: converge-in-pixel-pictures}
% \end{figure}

The set of finite graphs endowed with the cut metric gives rise to a metric space, and the completion of this space is the \emph{space of graphons}. Graphons are represented by bounded {symmetric} Lebesgue measurable functions $\FW: [0,1]^2 \rightarrow [0,1]$, which can be interpreted as weighted graphs on the node set $[0,1]$. 
%%%%%%%%%%%%
%
%%%%%%%%%%%%%%
We note that in some papers, for instance \cite{borgs2014lp}, the word "graphon" refers to symmetric, integrable functions from $[0,1]^2$ to $\BR$.   In this paper, unless stated otherwise, the term "graphon" is used to refer to  functions  $\mathbf{W_1}: [0,1]^2\rightarrow [-1,1]$ and  $\ESO$ denotes the space of graphons. Let $\ESZ$ represent the space of all graphons satisfying $\mathbf{W_0}: [0,1]^2\rightarrow [0,1]$ and let $\ES$ denote the space of all symmetric measurable functions  $\mathbf{W}: [0,1]^2\rightarrow \BR$. 

The cut norm of a graphon $\FW \in \ESO$ is then defined as
 \begin{equation}
 \| \mathbf{W} \|_{\Box}=\sup_{M,T\subset [0,1]}|\int_{M\times T}\mathbf{W}(x,y)dxdy|
 \end{equation}
 with the supremum taking over all measurable subsets $M$ and $T$ of $[0,1]$.
Evidently, the following inequalities hold between norms on a graphon $\FW$:
\begin{equation} \label{equ: norm-inequalities}
	\|\mathbf{W}\|_{\Box}\leq \|\mathbf{W}\|_1\leq \|\mathbf{W}\|_2\leq \|\mathbf{W}\|_{\infty}\leq1,
\end{equation}
where the second to the forth norms are given by the corresponding $L^p$ norms on $\ESO$. 
 Denote the set of measure preserving bijections from $[0, 1] $ to $ [0, 1]$ by $S_{[0,1]}$.  The \emph{cut metric} between two graphons $\mathbf{V}$ and $\mathbf{W}$ is then given by
\begin{equation}
	\delta_{\Box}(\mathbf{W, V})=\inf_{\phi\in S_{[0,1]}}\|\mathbf{W}^{\phi} -\mathbf{V}\|_{\Box},
\end{equation}
where $\mathbf{W}^{\phi}(x,y)=\mathbf{W}(\phi(x),\phi(y))$.
We see that the cut metric $\delta_{\Box}(\cdot, \cdot)$ is given by measuring the maximum discrepancy between the integrals of two graphons over measurable subsets of $[0,1]$, then minimizing the maximum discrepancy over all possible measure preserving bijections. 
Strictly speaking the cut metric is not a metric since the distance between two distinct graphons under the cut metric can be zero (see e.g. \cite{borgs2008convergent}).  
 However, by identifying functions $\mathbf{V}$ and $\mathbf{W}$ for which $\delta_{\Box}(\mathbf{V, W})=0$, we can construct the metric space $\GSO$  which denotes  the image of $\ESO$  under this identification. Similarly we can construct $\GSZ$ from $\ESZ$ and $\GS$  from $\ES$. See \cite{lovasz2012large}.

The {\it{$L^2$ metric}} for  any graphons $\mathbf{W}$ and $\mathbf{V}$ is defined as
\begin{equation}
	\begin{aligned}
		d_{L^2}(\FW, \FV) & = \|\FW-\FV\|_2 \\
		&  =\left(\int_{[0,1]^2} |\mathbf{W}(x,y)- \mathbf{V}(x,y)|^2 dxdy\right)^{\frac{1}{2}}
	\end{aligned}
\end{equation}
and the {\it{$\delta_2$ metric}} as 	
\begin{equation}
	\begin{aligned}
	&\delta_2 (\mathbf{W, V})= \inf_{\phi\in S_{[0,1]} }d_{L^2}(\mathbf{W^{\phi}, V})=\inf_{\phi  \in S_{[0,1]}}\|\mathbf{W}^{\phi}- \mathbf{V}\|_2;
	\end{aligned}
\end{equation}
 similarly, the {\it{$L^1$ metric}} and the \emph{$\delta_1$ metric} are defined respectively as
\begin{align}
	&d_{L^1}(\mathbf{W, V})=\|\mathbf{W}- \mathbf{V}\|_1,\\
  &\delta_1 (\mathbf{W, V}) 
  =\inf_{\phi  \in S_{[0,1]}}\|\mathbf{W}^{\phi}- \mathbf{V}\|_1.
\end{align} 

For any two graphons $\mathbf{W}$ and $\mathbf{V}$ the following inequalities hold immediately: 
\begin{equation} \label{equ:distance-inequality}
 \delta_{\Box}(\mathbf{W, V}) \leq \delta_1(\mathbf{W, V}) \leq \delta_2(\mathbf{W, V})\leq d_{L^2}(\mathbf{W, V}).
\end{equation}
The $\delta_2$ (or $\delta_1$) metric and $\delta_{\Box}$ metric share the same equivalence classes under the measure preserving transformations \cite[Corollary 8.14]{lovasz2012large}. Clearly, 
the $\delta_2$ (or $\delta_1$) metric is also well defined on $\GSO$. 

\subsection{Compactness of the Graphon Space}\label{subsec: compactness of graphon space}
\begin{theorem}[\cite{lovasz2012large}] \label{thm: compactness in G0 graphon space}
The space $( \GSZ,\delta_{\Box})$ is compact.
\end{theorem} 

This remains valid if $\GSZ$ is replaced by any uniformly bounded subset of $\GS$  closed  in the cut metric \cite{lovasz2012large}.
\begin{theorem}[\cite{lovasz2012large}] \label{thm: compactness in G1 extended graphon space}
The space $(\GSO,\delta_{\Box})$ is compact.
\end{theorem}

 	Sets in $\GSO$ (or $\GSZ$)  compact with respect to the $\delta_2$ metric are compact with respect to the cut metric.
It follows immediately from (\ref{equ:distance-inequality}) and Theorem \ref{thm: compactness in G1 extended graphon space} (or Theorem \ref{thm: compactness in G0 graphon space}), if a graphon sequence is Cauchy in the $\delta_2$ metric then it is also a Cauchy sequence in the cut metric and under both metrics, the limits 
are identical in $\GSO$ (or $\GSZ$).

Define the $L^p$ closed ball in $\GS$ with radius $C>0$ as $\mathcal{B}_{L^p} (C ) := \{ \FW : \| \FW \|_p \leq C, \FW \in \GS \}$. 
\begin{theorem}[\cite{borgs2014lp}]
	 The space $(\mathcal{B}_{L^p} (C ),\delta_{\Box} )$ with $1<p\leq \infty $ is compact. %
\end{theorem}	

By compactness, infinite sequences of graphons will necessarily possess one or more sub-sequential limits under the cut metric. 

\subsection{Step Function Graphons}
A function $\mathbf{W} \in \ESO$ is called a {\it step function} if there is a partition $Q=\{Q_1,...,Q_k\}$ of $[0, 1]$ into measurable sets such that $\mathbf{W}$ is constant on every product set $Q_i \times Q_j$. 
The sets $Q_i$ are the \emph{steps} of $\mathbf{W}$. 

Graphons generalize weighted graphs in the following sense. 
For every weighted graph $G$ with $N$ nodes, a step function ${\FW_G} \in \ESO$ is given
by 
partitioning $[0,1]$ into $N$ measurable sets $Q_1,\cdots ,Q_N$ of measure $\mu(Q_i)=\frac{\alpha_i}{\alpha(G)} $ and defining 
\begin{equation}
  \FW_G(x,y) := \sum_{i=1}^N  \sum_{j=1}^N \mathds{1}_{Q_i}(x) \mathds{1}_{Q_j}(y) \beta_{ij}(G), ~~ (x,y) \in [0,1]^2,
\end{equation}
 where $\alpha_i$ denotes the node weight of $i^{th}$ node, $\alpha(G)=\sum_i\alpha_i$ and $\beta_{ij}(G)$ denotes the  weight of the edge from node $i$ to node $j$ (i.e., $\beta_{ij}(G)$ is the $ij^{th}$ entry in the adjacency matrix of $G$). Evidently the function $\FW_G$ depends on the labeling of the nodes of $G$.
 % (see {\cite{lovasz2012large}}). 

 We define the \emph{uniform partition} $P^{N}=\{P_1, P_2, ..., P_{N}\}$  of $[0,1]$  by setting $P_k=[\frac{k-1}{N}, \frac{k}{N}), k\in\{1, N-1\}$ and $P_N=[\frac{N-1}N,1]$. Then $\mu({P_i})=\frac1{N}$, $ i \in\{1,2,...,N\}$. Under the uniform partition, the step functions can be represented by the pixel diagram on the unit square. See \cite{lovasz2012large}.

\subsection{Graphons as Operators}
 A graphon $\mathbf{W} \in \ESO$ can be interpreted as an operator 
$\mathbf{W}: L^{2}[0,1] \rightarrow L^{2}[0,1].$
The operation on $\mathbf{v} \in L^2[0,1]$ is defined as follows:
\begin{equation} \label{equ: graphon operation on functions}
	[\mathbf{Wv}](x)=\int_0^1\mathbf{W}(x,\alpha)\mathbf{v}(\alpha)d\alpha.
\end{equation}
The operator product is then defined by
\begin{equation} \label{equ: graphon operation on graphons}
	[\mathbf{UW}](x,y)=\int_0^1\mathbf{U}(x,z)\mathbf{W}(z,y)dz,
\end{equation}
where $\mathbf{U,W} \in \ESO.$ See \cite{lovasz2012large} for more details. 
For simplicity of notation, $\FU\FW$ is used to denote the graphon given by the convolution in (\ref{equ: graphon operation on graphons}); similarly, $\FW \Fv$ denotes the function defined by (\ref{equ: graphon operation on functions}).
Note that if $\mathbf{U}\in \ESO$ and $\mathbf{W} \in \ESO$, then $\mathbf{UW}\in \ESO$, since for all $x, y \in [0,1]$
\begin{equation}
	\begin{aligned}
	 \left|[\mathbf{UW}](x,y)\right|%
	 &\leq  \int_0^1|\mathbf{U}(x,z)\mathbf{W}(z,y)|dz \leq 1.\\  %
	\end{aligned} 
\end{equation}
Consequently, the power $\mathbf{W}^n$ of an operator $\mathbf{W} \in \ESO$ is defined as
\begin{equation*}
	\mathbf{W}^n(x,y)=\int_{[0,1]^n}\mathbf{W}(x,\alpha_1)\cdots \mathbf{W}(\alpha_{n-1},y)d\alpha_1\cdots d\alpha_{n-1}
\end{equation*}
with $\mathbf{W}^n \in \ESO $ $( n \geq 1)$.
$\mathbf{W^0}$  is formally defined as the identity operator $\BI$ on  functions in $L^2{[0,1]}$, but we note that $\mathbf{W^0}$ is not a graphon.

Any $\FA \in \ESO$ gives a self-adjoint compact operator \cite{ShuangPeterCDC19W2} and hence has a discrete spectral decomposition as follows: 
\begin{equation} \label{equ: graphon-infinite-spectral-sum}
  \FA(x,y) = \sum_{\ell=1}^{\infty} \lambda_\ell \Ff_\ell(x) \Ff_\ell(y), \quad (x,y)\in[0,1]^2,
\end{equation}
where the convergence is in the $L^2{[0,1]^2}$ sense, $\{\lambda_1, \lambda_2,....\}$  is the set of eigenvalues (which are not necessarily distinct) with decreasing absolute values, and $\{\Ff_1, \Ff_2,...\}$ represents the set of the corresponding orthonormal eigenfunctions (i.e. $\|\Ff_\ell\|_2=1$, and $\langle \Ff_\ell, \Ff_k\rangle =0$ if $l\neq k$). 
The only accumulating point of the eigenvalues is zero \cite{lovasz2012large}, that is, $\lim_{\ell\rightarrow \infty} \lambda_\ell =0.$ %Every non-zero eigenvalue has finite multiplicity, here $\lambda_\ell$ (including multiplicity) 
This implies that $\FA$ can be approximated by a finite truncation of the spectral decomposition which preserves the most significant eigenvalues \cite{ShuangPeterCDC19W2}.

\subsection{The Graphon Unitary Operator Algebra}
It is evident that the operator composition defined in (\ref{equ: graphon operation on graphons}) above yields an operator algebra with a multiplicative  binary operation possessing the associativity, left distributivity, right distributivity properties and compatibility with the scalar field $\BR$, that is, for any $\FV, \FW, \FH$ in the vector space $L^2{[0,1]}^2$ and $a, b \in \BR,$
$$
\begin{aligned}
	& (\FV \FW) \FH = \FV(\FW \FH),\\
	&\FV(\FW+\FH)=\FV \FW+\FV \FH,\\
	&(\FW+\FH)\FV=\FW\FV+\FH\FV,\\
	&(a\FW)(b\FH)=(ab)\FW\FH.
\end{aligned}
$$
Thus we have an operator algebra $\mathcal{G_A}$ over the field $\BR$ acting on elements of $L^2[0,1]$ with operator multiplication 
as given in (\ref{equ: graphon operation on functions}).
 By adjoining  the identity element $\BI$ to the algebra $\mathcal{G_A}$ (see e.g. \cite{kadison2015fundamentals}) we obtain a unitary algebra $\mathcal{G_{AI}}$. The identity element $\BI$ is defined as follows:
for any $\FW \in L^2[0,1]^2$
\begin{equation}
[\BI\FW](x,y)= \int_0^1 \FW(z,y)  \delta(x,z) dz = \FW(x,y),
\end{equation}
where  $\delta(\cdot, z)dz$ is the measure satisfying
$
\int_0^1 u(z) \delta(x,z)  dz =u(x)
$
for all $u\in L^2[0,1]$,
and in particular
$\int_0^1 \delta(x, z) dz = 1$. 

The \emph{graphon unitary operator algebra} $\mathcal{G_{AI}}$ will be used in the definition of the graphon dynamical systems. 
 More specifically, 
we use the subset $\mathcal{G}^1_\mathcal{{AI}}:=\{(a\BI+ \FA): \FA \in 
\mathcal{G}^1_\mathcal{{A}}, a \in \BR\}$
where $\mathcal{G}^1_\mathcal{{A}}$ is the subset of $\mathcal{G}_\mathcal{{A}}$  that corresponds to $\ESO$.

\subsection{Graphon Differential Equations}

Let $X$ be a Banach space.
% A {linear operator} $\BA : \mathcal{D}(\BA) \subset X \rightarrow X$ is \emph{closed}  if $\{(x, \BA x): x\in \mathcal{D}(\BA) \}$ is closed in the product space $X\times X$ (see \cite{bensoussan2007representation}).
%
$\mathcal{L}(X)$ denotes the Banach algebra of all linear continuous mappings $\BT: X\rightarrow X$, endowed with the norm
   $\|\BT\|_{\text{op}} = \sup_{x\in X, \|x\|=1} \|\BT x\|.$
 %
% Definition of Strong Continuous Semigroups
A mapping ${\BS}:\BR \rightarrow \mathcal{L}(X)$ is said to be a {\it strongly continuous  semigroup} on $X$ if the following properties hold: %
\begin{enumerate}
	 \item ${\BS}(0)=\BI, ~ \BS(t+s)=\BS(t)\BS(s), \quad \forall t,s \geq 0$
	 \item for all $x\in X$, $\BS(\cdot)x$ is continuous on $\BR$.
\end{enumerate}
A \emph{uniformly continuous semigroup} is a strongly continuous semigroup $\BS$ such that
$
	\lim _{{t\to 0^{+}}}\left\|\BS(t)-\BI\right\|_{\text{op}}=0,
$
with $\|\cdot\|_{\text{op}}$ as the operator norm on a Banach space.
The \emph{infinitesimal generator} $\BA$ of a strongly continuous semigroup $\BS$ is the linear operator in $X$ defined by
				$\BA x=\lim_{t \rightarrow 0^+}\frac{1}{t}[\BS(t)x-x], \text{ for all }x\in \mathcal{D}(\BA),$
where $\mathcal{D}(\BA)=\{x \in X:  \text{ s.t. } \lim_{t\rightarrow 0^+}\frac{1}{t}\left[\BS(t)x-x\right] \text{ exists}\}.\\
$

 Let $\mathbf{A}:[0,1]^2\rightarrow [-1,1]$ be a graphon and $\alpha \in \BR$.  Hence $(\alpha \BI + \FA)$ is a bounded %and closed 
 linear operator from $ L^2[0,1]$ to $L^2[0,1]$. Following \cite{pazy1983semigroups}, $(\alpha \BI + \FA)$ is  the infinitesimal generator of the uniformly (and hence necessarily strongly) continuous semigroup 
$$
	\BS(t):=e^{(\alpha \BI + \FA)t}=\sum_{k=0}^{\infty} \frac{t^k\mathbf{(\alpha \BI + \FA)}^k}{k!}.
$$
Therefore, the initial value problem of the \emph{graphon differential equation}
\begin{equation}
	\mathbf{\dot{y}_t}=\mathbf{(\alpha \BI + \FA)y_t},\quad \mathbf{y_0} \in L^2[0,1] \label{Operator-Differential-Equation}
\end{equation}
has a  solution given by 
$
	\mathbf{y_t}=e^{(\alpha \BI + \FA)t}\mathbf{y_0}.
$

\begin{lemma}\label{lem:seperation-graphon-exp}
 Let $\mathbf{A}:[0,1]^2\rightarrow [-1,1]$ be a graphon and $\alpha \in \BR$. 
Then  $ e^{(\alpha \BI+ \FA)t} = e^{\alpha t} e^{\FA t}$ holds for all  $t\geq 0$.
\end{lemma}

Readers can readily check this and hence we omit the proof. 
% \begin{proof}
%   Let $\Phi(t)=  e^{(\alpha \BI+ \FA)t} -e^{\alpha t} e^{\FA t}$. %Based on the definition of uniformly continuous the semigroups for bounded linear operators, we have 
%   Then we have
%   \begin{equation}
%   \begin{aligned}
%      \dot{\Phi}(t) &= (\alpha \BI+ \FA)e^{(\alpha \BI+ \FA)t} -\alpha e^{\alpha t} e^{\FA t} - e^{\alpha t}\FA e^{\FA t}\\
%      & = (\alpha \BI+ \FA)\big(e^{(\alpha \BI+ \FA)t} -  e^{\alpha t} e^{\FA t}\big)\\
%      & = (\alpha \BI+ \FA) \Phi(t).
%   \end{aligned}
%   \end{equation}
%   Since $\Phi(0)= \BI -\BI =0$, we obtain $\Phi(t)=0$ for all $t\in R$ and hence the result in \eqref{equ:exponential-seperation}. 
% \end{proof}

\begin{lemma}\label{lem:exp-of-I+A}
   Consider any $\Fu\in L^2[0,1]$,  any $\BA \in \mathcal{G}^1_\mathcal{{AI}}$, and any $t\in[0,T]$. Let $\BA = \alpha \BI  + \FA$, $\FA \in \ESO$, then the following holds
   \begin{equation}
     \left\|e^{\BA t} \Fu \right\|_2  \leq  e^{t\left(\alpha +\|\FA\|_\textup{op}\right)} \|\Fu\|_2 \leq  e^{t\left(\alpha +\|\FA\|_2\right)} \|\Fu\|_2.
   \end{equation}
\end{lemma}

This result is immediate from Lemma \ref{lem:seperation-graphon-exp} and Lemma \ref{lem:graphon-exp-norm}.

\begin{theorem}[Appendix \ref{sec:Proofs for Graphon System Properties}] \label{theorem: convergence in operator} 

For any  $\alpha_N, \alpha, t \in \BR$, any $\FA_\FN, \FA_* \in \ESO$, and any $ \mathbf{x} \in  L^2[0,1]$, the following holds: 
\begin{align}
   &\left\|e^{\FA_\mathbf{N}t}\Fx-e^{\FA_*t}\Fx \right\|_2 %
    \leq 
      te^t \|\FA_{\Delta}^{\mathbf{N}}\|_\textup{op}\|\Fx\|_2, \label{equ:exp-of-A} \\[3pt]
    &\left\|e^{(\alpha_N \BI + \FA_\FN)t}\Fx - e^{(\alpha \BI + \FA_*)t}\Fx\right\|_2 \leq te^{(\alpha_N+1) t} \|\FA_{\Delta}^\mathbf{N}\|_\textup{op}\|\Fx\|_2 \nonumber\\
  &\qquad \qquad \qquad \qquad    + |\alpha-\alpha_N| t  e^{(L_\alpha +\|\FA_*\|_\textup{op})t}\left\|\Fx\right\|_2, \label{equ:exp-of-A+I} ~
\end{align}
where $\FA_{\Delta}^\FN = \FA_\FN -\FA_*$ and $L_\alpha= \max\{|\alpha|, |\alpha_N|\}$. 
Furthermore if  a sequence of graphons $\{\mathbf{A_N}\}_{N=1}^{\infty}$ and that of real numbers $\{\alpha_N\}_{N=1}^{\infty}$ converge as follows
\begin{equation}\label{eq:convergence-op-graphon}
  \lim_{N\rightarrow \infty} \|\FA_\FA-\FA_*\|_\textup{op} = 0 \quad \text{and} \quad \lim_{N\rightarrow\infty}|\alpha_N-\alpha|=0,
\end{equation}
then for any $\Fx\in L^2[0,1]$ and any $T>0$,
\begin{align}
  &\lim_{N\rightarrow\infty} \sup_{t\in[0,T]} \left\|e^{\mathbf{A_N}t}\Fx- e^{\mathbf{A_*}t}\Fx\right\|_2 =0,\label{equ:convergence-exp-1}\\
  & \lim_{N\rightarrow \infty}\sup_{t\in[0,T]}\left\|e^{(\alpha_N \BI + \FA_\FN)t}\Fx - e^{(\alpha \BI + \FA_*)t}\Fx\right\|_2= 0.\label{equ:convergence-exp-2}
\end{align}
\end{theorem}

The operator norm in \eqref{equ:exp-of-A}, \eqref{equ:exp-of-A+I} and \eqref{eq:convergence-op-graphon} can be replaced by the stronger $L^2[0,1]$ norm since $\|\FW\|_{\text{op}}\leq \|\FW\|_2$ for any $\FW\in \ESO$ (see Lemma \ref{lem:operator-norm-and-L2-norm}). 

\section{Network Systems and Their Limit Systems}
\subsection{Network System Models}
\begin{definition}[Network System]
  Consider an interlinked network ${S}^{N}$ of linear (symmetric) dynamical %control
 subsystems $\{S_i^N; 1\leq i \leq N \}$, each with an $n$ dimensional state space. 
 The subsystem $S_i^N$ at the node $V_i$ in the network $G_N(V,E)$ has  interactions with $S_j^N, 1\leq j \leq N,$ specified as below:
 \vspace{-0.2cm}
 \begin{equation} \label{equ:network-sys}
 S_i^N: ~
 \begin{aligned}
    &\dot{x}^i_t=  \alpha_N x^i_t + \frac{1}{nN}\sum_{j=1}^{N} {A}_{ij} x^j_t+\beta_N u^i_t+ \frac1{nN} \sum_{j=1}^N {B}{_{ij}} u^j_t, \quad \\
  &x^i_t, u^i_t \in \BR^n,  i\in\{1,...,N\},
  \end{aligned}
 \end{equation}  
  with ${A}_{N}= [{A}_{ij}], {B}_{N} =[{B}_{ij}] \in \BR^{nN\times nN}$ as the (symmetric) block-wise adjacency matrices of $G_N(V,E)$ and  of the input graph, %
  where $A_{ij} =[0]$ if $S^N_i$ has no connection to $S^N_j$ and similarity for $ B_{ij}$.  We call ${S}^{N}$ a \emph{network system}.  
\end{definition}

 Then the (symmetric) linear dynamics for the network system ${S}^N(A_N, B_N, G_N)$ can be represented by 
 \begin{equation}  \label{equ: averaging-dynamics}
 {S}^N: \quad 
 	\begin{aligned}
		& \dot{ {x}}_t=\alpha_N x_t + A_N\circ  {x}_t+ \beta_N u_t+B_N\circ  {u}_t,  \\
		&  {x}_t,  {u}_t\in \BR^{nN}, A_{N}, B_{N} \in \BR^{nN\times nN},
	\end{aligned}	
\end{equation}
where   $\circ$ denotes the so called averaging operator given by
$ %\label{equ:averaing-over-vector-multiplication}
	A_N\circ  {x}= \frac{1}{(nN)}A_N  {x}
$.
Let $\mathcal{S} = \times_{N=1}^{\infty}\mathcal{S^N}$ where $\mathcal{S^N}= \cup_{A_N, B_N, G_N} S^N(A_N, B_N, G_N).$
For simplicity, we require the elements of $A_N$ and $B_N$ to be in $[-1,1]$ for each $N$ (note that in general $A_N$ and $B_N$ have elements that are uniformly bounded real numbers  for which case we would achieve similar results).  In addition, we note that if we take the supremum norm on vectors in $\BR^{nN}$, i.e. $\|x \|_{\infty}=\sup_i|x_i|$, and the corresponding $\circ$ operator norm of $A$, i.e. $\|A \|_{\text{op}\infty}= \sup_{\|x\|_{\infty}\neq 0} \frac{\|A\circ x\|{\infty}}{\|x\|_{\infty}}$, then   $\|A \|_{\text{op}\infty}\leq 1.$    
    
\subsection{Network Systems Represented by Step Functions}

Let $\left\{(\alpha_N I+ A_N; \beta_N I+ B_N)\right\}_{N=1}^{\infty} \in \mathcal{S} $  be a sequence of systems with the node averaging dynamics each of which is described according to (\ref{equ: averaging-dynamics}).
Let $|A_{Nij}|\leq1$ and $|B_{Nij}|\leq 1$ for all $i,j \in\{1,..., nN\}$. 
Let  $\SA, \SB \in \ESO$ be the step functions corresponding one-to-one to $A_N$ and $B_N$; these are specified using the uniform partition $P^{nN}$ of $[0,1]$  by the following \emph{matrix to step function mapping $M_G$}: 
\begin{equation} \label{equ: AN-matrix-to-stepfundtion}
\begin{aligned}
 	\SA(x,y) := \sum_{i=1}^{nN} \sum_{j=1}^{nN} \mathds{1}_{P_i}(x)  \mathds{1}_{P_j}(y)
  A_{Nij}, \quad \forall (x,y) \in [0,1]^2
\end{aligned}
\end{equation}
and similar for $\SB$. In fact, the step function $\SA$ can represent a set of matrices of different sizes given by $\{A_N \otimes \text{one}_m, m\in \mathds{N}\}$ where $\text{one}_m$ is the $m\times m$ matrix of ones. 

Define a {\it{piece-wise constant (PWC) function}} on $\BR$  to be any function of the form $\sum_{k=1}^l\alpha_k\mathds{1}_{I_k}(\cdot)$ where $\alpha_1,...,\alpha_l$ are real numbers and each $I_k$ is a bounded interval (open, closed, or half-open). Let $L^2_{pwc}[0,1]$ denote the space of piece-wise constant $L^2{[0,1]}$ functions under the uniform partition $P^{nN}$.

$\Sut\in L_{pwc}^2[0,1]$ corresponds one-to-one to $u_t\in \BR^{nN}$ via the following  \emph{vector to PWC function mapping} also denoted by $M_G$: %for all $i \in\{1,...,nN\}$,
\begin{equation} \label{equ:u_s and u relation}
\Sut(\alpha):= \sum_{i=1}^{nN} \mathds{1}_{P_i}(\alpha)u_t(i), \quad \forall \alpha \in [0,1],
\end{equation}
and similarly $\Sxt\in L_{pwc}^2[0,1]$  corresponds one-to-one to $x_t\in \BR^{nN}$.
\begin{lemma}[Appendix \ref{sec:Proofs for Graphon System Properties}]\label{lem: Network-Stepfunction-Graphon}
The trajectories of the system in (\ref{equ: averaging-dynamics}) correspond one-to-one under the mapping $M_G$ to the trajectories of the system 
\begin{equation}\label{equ:step-function-dynamical-system}
	\begin{aligned} 
	&\mathbf{\dot{x}^{[N]}_{t}}= (\alpha_N\BI +\SA) \Sxt+(\beta_N \BI+\SB)\Sut, \\
	& \Sxt, \Sut \in L_{pwc}^2[0,1], \SA, \SB \in \ESO\subset \mathcal{G}^1_\mathcal{{AI}}
	\end{aligned}
\end{equation}
with graphon operations defined according to (\ref{equ: graphon operation on functions}). 
\end{lemma}
Since the system in \eqref{equ:step-function-dynamical-system} corresponds to a network system $S^N$ in \eqref{equ: averaging-dynamics}, we also refer to it as a \emph{network system}. We use $(\BSA;\BSB)$ to denote the network system in \eqref{equ:step-function-dynamical-system} where $\BSA= \alpha_N\BI +\SA$ and $\BSB= \beta_N\BI +\SB$. 

% The system model in \eqref{equ:step-function-dynamical-system} can represent a class of network systems of different dimensions in \eqref{equ: averaging-dynamics} if we ignore the difference cased by the partitions for states and inputs. The equivalent class of systems is then given by  
% \begin{equation}\label{equ:equivalent-class-of-systems}
%   \left\{(\alpha_N I + A_N \otimes \text{one}_p;\beta_N I + B_N \otimes \text{one}_q ): p, q\in \BN \right\}.
% \end{equation}

\subsection{Limits of Sequences of Network Systems}
A sequence of network systems with node averaging dynamics in \eqref{equ: averaging-dynamics} can be represented by the sequence of systems %
$\{(\alpha_N \BI +\SA;\beta_N \BI +\SB)\in \mathcal{G}^1_\mathcal{{AI}}  \times \mathcal{G}^1_\mathcal{{AI}} \}_{N=1}^{\infty}$ in \eqref{equ:step-function-dynamical-system}.

\begin{definition}[System Sequence Convergence] \label{def:system-convergence}
	A sequence of systems $\left\{(\alpha_N \BI +\SA;\beta_N \BI +\SB) \in \mathcal{G}^1_\mathcal{{AI}}  \times \mathcal{G}^1_\mathcal{{AI}} \right\}_{N=1}^{\infty}$ is convergent if the following two conditions hold
	\begin{enumerate}[1)]
		\item there exist $\alpha, \beta \in \BR$ such that 
    \begin{equation*}
      \lim_{N\rightarrow \infty} \alpha_N =\alpha ~~  \text{and} ~~  \lim_{N\rightarrow \infty} \beta_N =\beta;
    \end{equation*}
    \item there exist $\FA, \FB \in \ESO$, 
    \begin{equation*}%\label{equ:convergne-A-and-B}
      \lim_{N\rightarrow \infty} \|\FA-\SA\|_{\textup{op}} =0 ~~ \text{and}~~  \lim_{N\rightarrow \infty} \|\FB-\SB\|_{\textup{op}}=0.
    \end{equation*}
	\end{enumerate}	
\end{definition}

The limit system is represented by $( \BA; \BB) $ where $\BA = \alpha \BI +\FA$ and 
 $ \BB= \beta \BI  + \FB$. %

Since any $\FW \in \ESO$ defines a self-adjoint and compact operator, the maximum absolute value of eigenvalues of $\FW$ equals to the operator norm
\cite[Theorem 12.31]{rudin1991functional}, that is, 
  $\|\FW\|_{\text{op}} = \max_\ell |\lambda_\ell|.
  $
Furthermore, the following inequalities between the cut norm and the operator norm hold \cite{janson2010graphons,parise2018graphon}:
\begin{equation}\label{equ:cut-norm-operator-norm}
  \|\FW\|_\Box \leq \|\FW\|_{\text{op}} \leq \sqrt{8\|\FW\|_\Box} . 
\end{equation}
By Lemma \ref{lem:operator-norm-and-L2-norm}, the following inequality also holds:
\begin{equation}\label{equ:op-l2}
   \|\FW\|_{\text{op}} \leq \|\FW\|_2.
\end{equation}

Based on \eqref{equ:cut-norm-operator-norm} and \eqref{equ:op-l2}, we obtain that the convergence of a sequence of graphons in $\|\cdot\|_2$ or  $\|\cdot\|_\Box$ implies its convergence in $\|\cdot\|_{\textup{op}}$.  Under certain extra conditions, the convergence of a sequence of graphons under the cut norm implies its convergence under the $L^2[0,1]$ norm \cite[Corollary 1.1]{szegedy2011limits}.

Let the graphon sequences $\{ \SA\}$ and $\{\SB\}$ be Cauchy sequences of step functions in $L^2[0,1]^2$. Due to the completeness of $L^2{[0,1]^2}$, the respective graphon limits $\mathbf{A}$ and $\mathbf{B}$ exist and these will then necessarily  be the limits in the cut metric (see Section \ref{subsec: compactness of graphon space} and \cite{lovasz2012large}). 

\section{The Limit Graphon System and Its Properties}
%\subsection{Graphon Systems}
We follow \cite{bensoussan2007representation} and specialize both the Hilbert space of states $H$, and that of controls $U$ appearing there, to the space $L^2[0,1]$.
Let 
$L^2([0,T];L^2[0,1])$ denotes the Hilbert space of equivalence classes of strongly measurable (in the B\"ochner sense \cite[p.103]{showalter2013monotone}) mappings $[0,T] \rightarrow L^2[0,1]$ that are integrable with norm
$
\| \Ff \|_{L^2([0,T];L^2[0,1])} =\left[ \int_0^T \|\Ff(s)\|^2_2 ds\right]^{\frac{1}{2}}.
$

\begin{definition}[Graphon Systems]
  We formulate an infinite dimensional linear system, which we call a \emph{graphon system} $(\BA; \BB)$,  as follows: 
\begin{equation} \label{equ: infinite-system-model}
  \begin{aligned}
  &\dot{\Fx}_\Ft= \BA \Fx_\Ft + \BB \Fu_\Ft,  \quad \mathbf{x_0} \in L^2[0,1],   \end{aligned}
\end{equation}
where $\BA, \BB \in \mathcal{G}^1_\mathcal{{AI}}$, and are hence bounded operators on $L^2[0,1]$, $\mathbf{x_t} \in L^2[0,1]$ is the system state at time $t$ and $\mathbf{u_t} \in L^2[0,1]$ is the control input at time $t$.
\end{definition}
Notice that the network system in \eqref{equ:step-function-dynamical-system} is a special case of the graphon system in \eqref{equ: infinite-system-model}. 

%\subsection{Uniqueness of the Solution}
A  solution {$\mathbf{x_{(\cdot)}} \in L^2([0,T];L^2[0,1])$  is a {\it (mild) solution} of  \eqref{equ: infinite-system-model} if 
$  \mathbf{x_t}=e^{(t-a)\BA}\mathbf{x_a}+\int_a^te^{(t-s)\BA}\BB\Fu_\Fs ds$
for all $a$ and $t$ in $[0,T]$, taken to be $a\leq t$ (see \cite{bensoussan2007representation}).
Let $C([0,T];L^2[0,1])$  denote the set of continuous mappings from $[0,T]$ to $L^2[0,1]$.
\begin{proposition}\label{prop: uniqueness of the solution to infinite dim system}
The graphon system  in  (\ref{equ: infinite-system-model}) has a unique  solution $\mathbf{x}\in C([0,T];L^2[0,1])$ for all $\mathbf{x_0} \in L^2[0,1]$ and all $\mathbf{u}\in L^2([0,T]; L^2[0,1])$.
\end{proposition}
\begin{proof}
$\BA$ as a bounded linear operator generates a uniformly continuous semigroup and $\BB \in \mathcal{G}^1_\mathcal{{AI}}$ as a bounded linear operator forms a continuous linear mapping from control space $L^2[0,1]$ to the state space $L^2[0,1]$. Hence following \cite[p.385]{bensoussan2007representation}, the system (\ref{equ: infinite-system-model}) has a unique  solution $\mathbf{x}\in C([0,T];L^2[0,1])$ for all $\mathbf{x_0} \in L^2[0,1]$ and all $\mathbf{u}\in L^2([0,T]; L^2[0,1])$.
\end{proof}
%\subsection{Controllability}
A system ${(\BA;\BB)}$ is \emph{exactly controllable} on $[0,T]$ if for any  initial state $\mathbf{x}_0 \in L^2{[0,1]}$  and any target state $\mathbf{x}_f \in L^2{[0,1]}$, there exists a control $\Fu \in L^2([0,T];  L^2{[0,1]} )$ driving the system from $\mathbf{x}_0$ to $\mathbf{x}_f$, i.e.
$
	\Fx_T=\mathbf{x}_f
$ with $\Fx_T=e^{\BA T}\mathbf{x_0} + \int_0^T e^{\BA(T-t)}\BB\mathbf{u}_\Ft dt.$
A system $\mathbf{(\BA;\BB)}$ is \emph{approximately controllable} on $[0,T]$ if for any  initial state $\mathbf{x}_0 \in L^2{[0,1]}$, any target state $\mathbf{x}_f \in L^2{[0,1]}$ and any $\varepsilon >0$, there exists a control $\Fu \in L^2([0,T];  L^2{[0,1]})$ which drives the system state from $\mathbf{x}_0$ into the $L^2[0,1]$ $\varepsilon$-neighborhood of $\mathbf{x}_f$, i.e., $
  \|\Fx_T- \mathbf{x}_f \|_2 \leq \varepsilon.
$

 The \emph{controllability Gramian operator} $\mathbb{W}_T:L^2[0,1] \rightarrow L^2[0,1]$ is defined as 
\begin{equation}
	\begin{aligned}
		{\BW}_{T}:=&\int_{0}^{T} e^{\BA(t-s)}\BB \BB^\TRANS e^{\BA^\TRANS(t-s)}ds, \quad T>0.
	\end{aligned}
\end{equation}
A necessary and sufficient condition for exact controllability on $[0,T]$ is the uniform positive definiteness of $\BW_T$: 
\begin{equation}
  \langle \BW_Th,h\rangle\geq c_T\|h\|^2
\end{equation}
for all $h\in L^2[0,1]$, where  $c_T>0$ and $\|\cdot\|$ is  the $L^2[0,1]$ norm. The positive definiteness of the controllability Gramian operator $\BW_T$ as a kernel is equivalent to the approximate controllability of the corresponding system (see \cite{bensoussan2007representation,curtain1995introduction}).

Define the \emph{kernel space} (or \emph{null space})  of a linear operator $\BT$ on $L^2[0,1]$ as:
$
	\text{ker}(\BT) := \{x \in L^2[0,1]: \BT x=0 \}.
$
The \emph{spectrum} $\sigma(\BT)$ of a bounded linear operator $\BT$ on $L^2[0,1]$ is the set of all (complex or real) scalars  $\lambda$ such that $\BT-\lambda \BI$ is not invertible.  Thus $\lambda \in \sigma(\BT)$ if and only if at least one of the following two statements is true: 
\begin{enumerate}[(i)]
	\item The range of $\BT -\lambda \BI $ is not all of $L^2[0,1]$, i.e., $\BT-\lambda \BI$ is not onto.
	\item $\BT -\lambda \BI$ is not one-to-one.
\end{enumerate}
If (ii) holds, $\lambda$ is said to be an \emph{eigenvalue} of $\BT$; the corresponding eigenspace is $\text{ker}(\BT- \lambda \BI)$; each $x \in \text{ker}(\BT- \lambda \BI)$ (except $x=0$) is an \emph{eigenvector} of $\BT$; it satisfies the equation $\BT x= \lambda x$. See \cite{rudin1991functional}.

\begin{theorem}[Appendix \ref{sec:Proofs for Exact Controllability}] \label{thm:sufficient-condition-for-exact-controllability}%
			Let $\BA$ be an element in $\mathcal{G}^1_\mathcal{{AI}}$ and let $\BB$ be a  bounded  linear operator on $L^2[0,1]$.
			The linear system $(\BA; \BB)$ is exactly controllable on a finite time horizon $[0,T]$ if all the values in the spectrum of $\BB\BB^T$ are lower bounded by a strictly positive constant.
		\end{theorem}		
\begin{proposition}[Appendix \ref{sec:Proofs for Exact Controllability}] \label{prop:graphon-system-with-graphon-input-not-exact-controllable} %
			Let $\FA$ and $\FB$ be graphons in $\ESO$. Then  the linear system $(\FA; \FB)$ is not exactly controllable on any finite time horizon $[0,T]$. 
\end{proposition}	
	
The results in Theorem \ref{thm:sufficient-condition-for-exact-controllability} and Proposition \ref{prop:graphon-system-with-graphon-input-not-exact-controllable} generalize to the case where  the underlying graphons lie in any uniformly bounded subset of $\ES$, i.e., any set of symmetric measurable functions $\FW:[0,1]^2\rightarrow I$ where $I$ is a closed bounded interval in $\BR$.

\section{Graphon State-to-state Control of Network Systems}
\subsection{Approximation of $L^2[0,1]$ Functions}
\begin{theorem}[{{\cite[Theorem 13.23]{hewitt1965real}}}]  \label{theorem: hewitt1965real}
	Let $\mu$ be any measure on $\BR$ and $\mathcal{B}_{\mu}$ be the $\sigma$-algebra of $\mu$-measurable sets, and let $1\leq p<\infty$. Then piece-wise constant functions on $\BR$ form a dense subset of $L^p(\BR, \mathcal{B}_{\mu},\mu)$.
\end{theorem}

\begin{proposition}\label{prop:contraction-L2-function}
Let $\Fv \in L^2[0,1]$ be approximated by $\Fv^{\mathbf{[N]}} \in L^2_{pwc}[0,1]$ as follows:  for all $x \in Q_i$ and for all $i\in\{1,\ldots,nN\}$, 
\begin{equation} \label{equ:L2function-approx}
 \Fv^{\mathbf{[N]}}(x) = \frac1{\mu(Q_i)} \int_{Q_i} \Fv(\alpha) d\alpha
\end{equation}
with the partition $\{Q_1, Q_2,\ldots,  Q_{nN}\}$  of $[0,1]$ where $\mu(Q_i)$ denotes the measure of $Q_i$.
Then 
  $\|\Fv^{\mathbf{[N]}}\|_2 \leq \|\Fv\|_2.$
\end{proposition}
\begin{proof} Applying the Cauchy-Schwarz inequality yields
  \begin{equation}
  \begin{aligned}
      \|\Fv\|_2^2 &= \int_0^1 \Fv^2(x)dx = \sum_{i=1}^{nN}\int_{Q_i} \Fv^2(x)dx\\
      & \geq \sum_{i=1}^{nN} \frac1{\mu(Q_i)} \left[\int_{Q_i}\Fv(x)dx\right]^2\\
       & = \sum_{i=1}^{nN}  {\mu(Q_i)} \left[\frac1{\mu(Q_i)}\int_{Q_i}\Fv(x)dx\right]^2
       = \|\Fv^{\mathbf{[N]}}\|_2^2. 
  \end{aligned}
  \end{equation}~
\end{proof}

In  this paper we wish to approximate any control input $\mathbf{u}_\Ft \in L^2[0,1],  0 \leq t \leq T$,  through a piece-wise constant function in $ L^2[0,1]$ denoted by  $\Sut$. Specifically, the approximation of an input $\mathbf{u}_{\Ft}$ via the function $\Sut$ with the partition $Q=\{Q_1,Q_2,\cdots, Q_{nN}\}$ of $[0,1]$ will be specified as follows:  for all $Q_i, i \in \{1,2,\dots, nN\}$, 
\begin{equation} \label{equ: input-approximation}
				 {\Fu^\FN_{\Ft}}{(\alpha)}=\frac{1}{\mu(Q_i)}\int_{Q_i}\mathbf{u}_\Ft(\beta)d\beta,\quad \forall \alpha \in Q_i,
\end{equation}
where $\mu(Q_i)$ denotes the measure of $Q_i$.

\subsection{Limit Control for Network Systems}

\begin{theorem}[Appendix \ref{sec:Proofs for State-to-state Graphon Control}]\label{thm:mm-graphon-main}
  Consider the problem of driving the systems $(\BSA;\BSB)$ in \eqref{equ:step-function-dynamical-system} and $(\BA;\BB)$ in \eqref{equ: infinite-system-model} from the origin to some target state. Let $\BSA= \alpha_N \BI + \SA$, $\BSB= \beta_N \BI + \SB$, $\BA = \alpha \BI + \FA$ and $\BB = \beta \BI + \FB$.  Let ${\Fx_\FT(\Fu)}$ represent the terminal state of $(\BA;\BB)$ under control $\Fu$ and $\Fx^\FN_\FT(\Su)$ represent the terminal state of $(\BSA;\BSB)$ under control $\Fu^{[\FN]}$.
  Then for any $T>0$, 
there exists a control ${\Su}$ for $(\BSA; \BSB)$ approximating the control $\Fu$ for $(\BA;\BB)$ such that 
  \begin{equation} \label{equ:MM-terminal-difference}
    \begin{aligned}
   & \|{\Fx_\FT(\Fu)} - {\Fx^\FN_\FT(\Fu^{[\FN]})}\|_2 \\
   &\leq \left|\alpha- \alpha_N\right| (|\beta| + \|\FB\|_\textup{op}) \\
   & \qquad \qquad \qquad\int_0^T (T-t)e^{(L_\alpha+\|\FA\|_\textup{op})(T-t)}\left\|\Fu_\Ft\right\|_2 dt \\
     &+  \|\FA_{\Delta}^{\mathbf{N}}\|_\textup{op}(|\beta| + \|\FB\|_\textup{op}) \int_0^T e^{(\alpha_N+1)(T-t)} (T-t) \|\Fu_\Ft\|_2 dt\\
         &+  (|\beta-\beta_N|+ \|{\FB_{\Delta}^\FN}\|_\textup{op})\\
         & \qquad \qquad \qquad  \int_0^T e^{\left(\alpha_N + \left\|\SA\right\|_\textup{op}\right)(T-t)} \|\Fu_\Ft\|_2 dt \\
     &+ |\beta_N|\int_0^T e^{\left(\alpha_N + \|\SA\|_\textup{op}\right)(T-t)} \left\|\Fu_\Ft - \Sut\right\|_2 dt,
  \end{aligned}
  \end{equation}
   where $L_\alpha= \max\{|\alpha_N|,|\alpha|\}$,  ${\FA_{\Delta}^\FN}=\FA-\SA$,  ${\FB_{\Delta}^\FN}=\FB-\SB$,  and the control approximation is given in the following:  
  \begin{equation}\label{equ:mm-control-approximation}
    {\Sut}{(\alpha)}=nN\int_{P_i}\Fu_\Ft(\beta)d\beta ,
  \end{equation}
             for all $ \alpha \in P_i$, $t\in[0,T]$,
    with the uniform partition $P^{nN} =\{P_1, \cdots, P_{nN}\}$.
Furthermore, if a sequence of network systems $\{(\BSA; \BSB)\}_{N=1}^\infty$ converges to a graphon system $(\BA;\BB)$ as in Definition \ref{def:system-convergence}, then
  for any $\varepsilon >0$  there exists $N_\varepsilon>0$ such that  each $N \ge N_\varepsilon$, 
   \begin{equation} \label{equ:mm-terminal-convergence}
     \left\|{\Fx_\FT(\Fu)} - {\Fx^\FN_\FT(\Fu^{[\FN]})}\right\|_2 < \varepsilon.
   \end{equation}

 \end{theorem}

The control law ${u^N_{(\cdot)}}$ for the finite network system $({\alpha_N I + A_N; \beta_N I + B_N})$ is given by 
$$
	   {u^N_t}(i)={\Fu^{[\FN]}_{\Ft}}(\alpha),  \quad
	  \forall i\in\{1,...,nN\},   \forall \alpha \in P_i,  t\in[0, T].
$$
 Note that ${u^N}$ always exists by definition since the control approximation given in the definition (\ref{equ: input-approximation}) uses the same uniform partition as the step function approximation in the graphon space.

The operator norm in \eqref{equ:MM-terminal-difference} can be replaced by the $L^2[0,1]^2$ norm since $\|\cdot\|_{\textup{op}} \leq \|\cdot\|_2$.

\subsection{The Graphon State-to-state Control (GSSC) Strategy }
Consider the control problem of steering the states of each member of $\{(\alpha_N I +A_N; \beta_N I+ B_N)\}_{N=1}^{\infty} \in \mathcal{S}$ to each of a sequence of desired states $\{x^N_T \in \BR^{nN}\}_{N=1}^{\infty}$.  

The \textbf{Graphon State-to-state Control (GSSC) Strategy} consists of four steps:
\begin{enumerate}[\bf S.1]
	\it
	\item Let $\{(\alpha_N\BI+\SA;\beta_N \BI+ \SB) \in \mathcal{G}^1_\mathcal{{AI}} \times \mathcal{G}^1_\mathcal{{AI}}\}_{N=1}^{\infty}$ be  the sequence of graphon dynamical systems equivalent to  $\{(\alpha_N I +A_N; \beta_N I+ B_N)\}_{N=1}^{\infty} \in \mathcal{S}$ under the mapping $M_G$, and  assume  it converges to the graphon system $(\alpha\BI + \FA;\beta\BI+\FB) \in \mathcal{G}^1_\mathcal{{AI}} \times \mathcal{G}^1_\mathcal{{AI}}$ as in Definition \ref{def:system-convergence}. Let $\{ \Fx^\FN_\FT \in L^2[0,1]\}_{N=1}^{\infty}$ be the image of  $\{x^N_T \in \BR^{nN}\}_{N=1}^{\infty}$ under $M_G$, which is assumed to converge to some $\Fx^{\infty}_\FT \in L^2[0,1]$ in the $L^2[0,1]$ norm.
	\item Specify the corresponding state to state control problem $CP^{\infty}$ for $(\alpha\BI + \FA;\beta\BI+\FB) \in \mathcal{G}^1_\mathcal{{AI}} \times \mathcal{G}^1_\mathcal{{AI}}$  with $\Fx^{\infty}_\FT$  as the target terminal state  and choose a tolerance $\varepsilon>0$.
	\item Find a control law $\Fu^{\infty} := \{\Fu_\mathbf{\tau} \in L^2[0,1], \tau \in [0, T]\}$ solving   $CP^{\infty}$.
	
	\item Then generate the control law $\{\Su\}$ according to Theorems \ref{thm:mm-graphon-main} for which the convergence of 
	 $\{\Fx_\FT^\FN(\Fu^\FN)\}$ to $\Fx^{\infty}_\FT$  is guaranteed. 
Together with the assumed convergence of $\{ \Fx^\FN_\FT \in L^2[0,1]\}_{N=1}^{\infty}$ to $\Fx^{\infty}_\FT$, it yields $N_{\varepsilon}$  such that  $\Fx_\FT^\FN(\Fu^\FN)$ is within $\varepsilon$ of $\Fx^\FN_\FT$ for all $N\geq N_{\varepsilon}$ under the $L^2[0,1]$ norm.

\end{enumerate}

The notion of the effectiveness of the GSSC strategy for a sequence of network systems is defined to mean that (1) the terminal state is close to that achieved by the minimum energy control; (2) the computation for generating the control law is tractable.

The basic assumptions for the GSSC strategy are that (i) a sequence of finite network systems of interest converges to a limit graphon system (as in Definition \ref{def:system-convergence})
or a given instance of the network sequence can be closely approximated by a graphon system, 
  and (ii) the corresponding state-to-state control problem for the (limit) graphon system is tractable.% 

 These assumptions, together with the approximation theorem (i.e. Theorem \ref{thm:mm-graphon-main}), guarantee the effectiveness of the GSSC strategy for the finite network,  that is to say, the GSSC strategy can achieve the target terminal state with a small error by means of a tractable computation.% 
%%%%%%%%%%%%%%%%%
\subsection{Min-Energy State-to-state Control for Graphon Systems}
%%%
A specific control law which may be used in \textbf{S.2} of the GSSC strategy is described in this section.

Define the energy cost by the control over the time horizon $[0, T]$ as
$
 	J(\Fu) = \int_{0}^{T} \|\mathbf{u_{\tau}}\|^2 d\tau,  (0<T<\infty).
$
The objective is to drive the system  from some initial state $\Fx_\mathbf{0} \in L^2[0,1]$ to some target state $\Fx_\FT \in L^2[0,1]$ using minimum control energy. 
A function $\mathbf{u^*}\in L^2([0,T];L^2[0,1])$ is called an optimal control if
$
	J(\mathbf{u^*}) \leq J(\mathbf{u}),
$
for all $ \mathbf{u} \in L^2([0,T]; L^2[0,1])$ which drive the system from $\mathbf{x_0}$ to $\mathbf{x_T}$.
\begin{theorem}[Appendix \ref{sec:Inverse of the Controllability Gramian}]\label{theorem: existence of graphon Gramian inverse}
	If the graphon system $(\BA;\BB)$ in \eqref{equ: infinite-system-model} with $\BW_T$ as its graphon controllability Gramian operator is exactly controllable, then the inverse operator $ {\BW}_T^{-1}$ exists and is a bounded operator. 
\end{theorem} 

Assume the system $(\BA;\BB)$ is exactly controllable, then $\BW^{^{-1}}_T$ exists and the optimal control law that achieves the minimum energy control is given by
\begin{equation}\label{equ:mm-optimal-control-law}
	\mathbf{u^*_{\tau}}=  \BB^\TRANS e^{{\BA}^\TRANS(T-\tau)} {\BW}_T^{-1}(\mathbf{x_T}-e^{{\BA}(T)} \mathbf{x_0}),\quad \tau \in [0,T].
\end{equation}
The minimum energy for controlling the system in time horizon $[0,T]$ is 
\begin{equation}
	\begin{aligned}
		\|\mathbf{u}\|_{L^2([0,T];L^2[0,1])}^2 
		=  [\mathbf{x_\FT}-e^{{\BA}(T)}\mathbf{x_0}]^\TRANS{\BW}^{-1}_T[\mathbf{x_T}-e^{{\BA}(T)}\mathbf{x_0}].
	\end{aligned}
\end{equation}

Denote the spectral decomposition of $\FA$ is as follows
  $
  \FA (x,y) = \sum_{\ell\in I_\lambda} \lambda_\ell \Ff_\ell(x) \Ff_\ell(y),
  $
  where $\Ff_\ell$ is the normalized eigenfunction corresponding to the eigenvalue $\lambda_\ell$ and $I_\lambda$ is the index set for non-zero eigenvalues of $\FA$, which contains a countable number of elements \cite{lovasz2012large}. 
\begin{proposition}[Appendix \ref{sec:Inverse of the Controllability Gramian}]\label{prop:inverse of the controllability Gramian operator} 
Consider a graphon system $(\FA; \BI)$. 
	Then 
  \begin{enumerate}[1)]
  	\item  the controllability Gramian operator is given by 
		\begin{equation}
      \BW_T = T\BI + \sum_{\ell\in I_\lambda} \left(\frac1{2\lambda_\ell}[e^{2\lambda_\ell T}-1]-T \right) \Ff_\ell \Ff_\ell^\TRANS ;
    \end{equation}
	\item the inverse of the controllability Gramian operator for $(\FA;\BI)$ is given by
  \begin{equation}\label{equ:gramian-inverse}
    \BW_T^{-1}= \frac1T \BI - \frac1T \sum_{\ell \in I_\lambda} \frac{\frac1{2\lambda_\ell}[e^{2\lambda_\ell T}-1]-T}{\left(\frac1{2\lambda_\ell}[e^{2\lambda_\ell T}-1] \right)} \Ff_\ell \Ff_\ell^\TRANS.
  \end{equation}
  \end{enumerate}
  \end{proposition}

 Note that 
 $
 	\lim_{\lambda_\ell \rightarrow 0}\left(\frac1{2\lambda_\ell}[e^{2\lambda_\ell T}-1]-T \right) =0.
 $ 
 %
%To achieve state-to-state control of linear system in infinite dimensional state space  requires the system to be exactly controllable.  
We further note that approximate controllability is not sufficient to achieve state-to-state control since in that case the inverse operator $\BW_T^{-1}$ may not be bounded on certain subspaces in $L^2[0,1]$, and moreover the corresponding energy required would be unbounded.

\section{Graphon Linear Quadratic Regulation (LQR) of Network Systems}

\subsection{LQR Problems for Graphon Dynamical Systems}
 For finite $T>0$, consider the problem of minimizing the cost given by
\begin{equation} \label{equ:LQ-cost}
	J(\Fu)=\int_0^T\left( \| \BC \Fx_{\tau}\|^2_2+\| \Fu_{\tau}\|^2_2\right)d\tau+\langle \BP_0\Fx_\FT, \Fx_\FT\rangle
\end{equation}
over all controls $\Fu\in L^2([0,T]; L^2[0,1])$ subject to the system model constrains in (\ref{equ: infinite-system-model}). 
\begin{assumption}\label{ass:P-and-C}~
  $\BP_0 \in \mathcal{L}(L^2[0,1])$ is Hermitian and non-negative; $\BC\in \mathcal{L}(L^2[0,1])$
\end{assumption}

Finding the feedback control via dynamic programming consists of the two following standard steps \cite{bensoussan2007representation}:
\begin{enumerate}[1)]
  \item Solving the Riccati equation
\begin{equation}\label{equ: Riccati-Equation}
  \dot{\BP}=\BA^\TRANS \BP+\BP \BA- \BP\BB\BB^\TRANS \BP + \BC^\TRANS \BC, \quad \BP_0 \in \mathcal{L}(L^2[0,1]);
\end{equation}
\item 
Given the solution $\BP$ to the Riccati equation,  the optimal control $\Fu^*$ is given by 
\begin{equation} \label{equ: feedback-control-law}
  \Fu^*_\Ft=-\BB^\TRANS \BP_{T-t} \Fx^*_\Ft, \quad t\in [0, T]
\end{equation}
and 
the optimal trajectory $\Fx^*$  is then the solution to the closed loop equation
\begin{equation} \label{equ: Closed-Loop-System}
  \begin{aligned}
   \dot{\Fx}_\Ft=\BA \Fx_\Ft- \BB \BB^\TRANS \BP_{T-t} \Fx_\Ft, ~
  t\in [0,T], \Fx_0 \in L^2[0,1].
  \end{aligned}
\end{equation}
\end{enumerate}
Let 
$
  \Sigma(L^2[0,1])=\big\{\BT\in \mathcal{L}(L^2[0,1]): \BT  \textrm{ is Hermitian}\big\}
$
and 
\begin{equation}
  \begin{aligned}
  &\Sigma^+(L^2[0,1]) \\
  &=\big\{\BT\in \Sigma(L^2[0,1]): \langle \BT \Fv, \Fv\rangle \geq 0, ~ \forall \Fv \in L^2[0,1] \big\}.
\end{aligned}
\end{equation}
Denote the topological space of all strongly continuous mappings $\mathbb{F}: I \rightarrow \Sigma(L^2[0,1])$ endowed with strong convergence (see \cite{bensoussan2007representation}) by $C_s(I; \Sigma(L^2[0,1]))$ where $I$ denotes a compact interval, that is, the convergence of $\mathbb{F}_N$ to $\mathbb{F}$ in $C_s(I; \Sigma(L^2[0,1]))$  is defined by 
\begin{equation}
  \forall \Fv \in L^2[0,1], \quad \lim_{N\rightarrow\infty} \sup_{t\in I}\|\mathbb{F}_N(t)\Fv - \mathbb{F}(t)\Fv\|_2 =0.
\end{equation}

\begin{proposition}[{{\cite[Part IV-1]{bensoussan2007representation}}}] \label{prop:Riccati-Sol}
Under \textup{Assumption \ref{ass:P-and-C}}, there exist a unique solution to the Riccati equation  (\ref{equ: Riccati-Equation}) in $C_s([0,T]; \Sigma^{+}(L^2[0,1]))$ and an optimal solution pair $(\Fu^*, \Fx^*)$ to (\ref{equ: feedback-control-law}) and (\ref{equ: Closed-Loop-System}) where $\Fx^* \in C([0,T];L^2[0,1])$ and $\Fu^* \in L^2([0,T];L^2[0,1])$.
\end{proposition}
\subsection{The Graphon-Network LQR (GLQR) Strategy}
Consider the control problem of regulating the states of each member of $\{(\alpha_N I + A_N; \beta_N I + B_N)\}_{N=1}^{\infty} \in \mathcal{S}$.

The \textbf{Graphon-Network LQR (GLQR) Strategy} is as follows:
\begin{enumerate}[\bf S.1]
	\it
	\item  Let $\{(\BSA; \BSB) \in \mathcal{G}^1_\mathcal{{AI}} \times \mathcal{G}^1_\mathcal{{AI}} \}_{N=1}^{\infty}$ be  the sequence of equivalent representation of network systems  $\{(\alpha_N I + A_N; \beta_N I+ B_N)\}_{N=1}^{\infty} \in \mathcal{S}$ under the mapping $M_G$ and  assume that it converges to the graphon system $(\BA;\BB) \in \mathcal{G}^1_\mathcal{{AI}} \times \mathcal{G}^1_\mathcal{{AI}}$ as in Definition \ref{def:system-convergence}.
	\item Define the linear quadratic cost for $(\BA;\BB)$ as 
\begin{equation}
  J(\Fu)= \int_0^T\left[\big\|\BC \Fx_\tau\big\|^2_2+\big\|\Fu_\tau\big\|^2_2\right]d\tau + \langle \BP_0 \Fx_\FT, \Fx_\FT\rangle
\end{equation}
 and  the linear quadratic cost for $(\BSA;\BSB)$ as 
\begin{multline}
  J(\Su)  = \int_0^T\left[\big\|\BSC \Sxt\big\|^2_2
        +\big\|\Sut\big\|^2_2\right]dt \\
        +\langle \BSPZ \SxT, \SxT\rangle
\end{multline}
where it is assumed that $\BSC \rightarrow \BC$ and $\BSPZ \rightarrow \BP_0$ in the strong operator sense.   Solve the infinite dimensional Riccati equation for $(\BA;\BB)$ to generate the solution $\BP$. 

	\item Approximate $\BP$ to generate $\BAP$ and hence the control law 
   $\Sut= -\BSB^\TRANS \BAPt{T-t} \Sxt$
   for $(\BSA;\BSB)$. 
\end{enumerate}

 Parallel to the state-to-state control problem, we take the notion of the effectiveness of the GLQR strategy for a sequence of network systems to be that (1) the regulation cost and the state trajectory are close to those achieved by the optimal LQR control; (2) the computation for generating the control law is tractable. 

 In analogy with the state-to-state control problem, the basic assumptions for the GLQR strategy are that the sequence of finite network systems converges to a limit graphon system (as in Definition \ref{def:system-convergence}) or that a given network system can be closely approximated by a graphon system, and that the corresponding LQR problem for the (limit) graphon system is tractable.

These assumptions, together with Theorem \ref{thm: Convergence of the States}, guarantee the effectiveness of the GLQR strategy for the finite network systems that are sufficiently close to the limit graphon system for sufficiently large node cardinality.

\subsection{Control Law Approximations}

 By approximating the Riccati equation solution $\BP$ for $(\BA;\BB )$ we can generate $\BAP$ that provides the control law for the finite dimensional network system:
 \begin{equation}
   \Sut= -\BSB^\TRANS\BAPt{(T-t)}\Sxt .
 \end{equation}

 Consider the strongly continuous linear operator $\BP$ in $C_s([0,T]; \Sigma^+(L^2[0,1]))$. Its approximation is given by
\begin{equation}\label{equ:approximate-Riccati-Solution}
  \begin{aligned}
    &\BAPt{t}(x,y) =\frac{\langle \mathbf{1}_{Q_i}, \BP_t \mathbf{1}_{Q_j} \rangle}{\mu(Q_i) \mu(Q_j)},
    & \forall t \in[0,T], \quad  \forall (x,y) \in Q_i \times Q_j, 
  \end{aligned}
\end{equation} 
where $\{Q_1, Q_2,\ldots,  Q_{nN}\}$ forms a partition of $[0,1]$ and $\mu(Q_i)$ represents the length of the interval $Q_i$. In the case of uniform partition, $\mu(Q_i)=\frac1{nN}$.

\begin{lemma}\label{lem:uniform-converge-Approx} % 
Let $\BAP$ be generated by the step function approximation of $\BP$ via the $N$ uniform partition of $[0,1]$ according to \eqref{equ:approximate-Riccati-Solution}. Then
\begin{equation}\label{equ:strong-convergence-Riccati-Aprox}
  \lim_{N\rightarrow \infty} \BAP = \BP, \quad \textup{ in } C_s([0,T]; \Sigma(L^2[0,1])).
\end{equation}
\end{lemma}
\begin{proof}
Consider an arbitrary function $\Fv\in L^2[0,1]$. Based on the definition of the step function approximation in \eqref{equ:approximate-Riccati-Solution}, for any $t\in[0,T]$,  $\BAPt{t} \Fv $ is the piece-wise-constant function approximation  of $\BP_{t} \Fv \in L^2[0,1]$ as in \eqref{equ:L2function-approx} and by Proposition \ref{prop:contraction-L2-function}, $\|\BAPt{t} \Fv\|_2\leq \|\BP_{t} \Fv\|_2$. Since the Riccati equation \eqref{equ: Riccati-Equation} over the closed bounded interval $[0,T]$ has a solution $\BP$ (see Proposition \ref{prop:Riccati-Sol}), for any $\Fv \in L^2[0,1]$
there exist $C_v>0$ such that
\begin{equation}
   \left\|\BP_{t_0}\Fv-\BP_t\Fv\right\|_2 \leq C_v|t_0-t|,  \quad \forall t_0, t \in [0,T].
 \end{equation} 
Note that $(\BAPt{t_0}-\BAPt{t})\Fv \in L^2[0,1]$ is an approximation of  $(\BP_{t_0}-\BP_t)\Fv \in L^2[0,1]$ following \eqref{equ:L2function-approx}. Therefore by the contraction property in Proposition \ref{prop:contraction-L2-function},  we obtain 
\begin{equation}
   \left\|\BAPt{t_0}\Fv-\BAPt{t}\Fv\right\|_2 \leq C_v|t_0-t|,  \quad \forall t_0, t \in [0,T],
 \end{equation} 
and hence the sequence of functions $\{\BAPt{(\cdot)}\Fv\}_{N=1}^\infty$ is equicontinuous (see e.g. \cite[p.43]{rudin1991functional}). 
Furthermore, $\BAPt{(\cdot)} \Fv $ and $\BP_{(\cdot)} \Fv$ are continuous functions defined over the closed bounded time interval $[0,T]$. Hence by the Arzel\`a-Ascoli Theorem, 
\begin{equation}
  \forall t \in [0,T], \quad \lim_{N\rightarrow \infty}  \left\| \BAPt{t} \Fv - \BP_{t} \Fv \right\|_2 =0 
\end{equation}
implies that, 
for any $\Fv\in L^2[0,1]$,
\begin{equation}
  \lim_{N\rightarrow \infty} \sup_{t\in[0,T]} \left\| \BAPt{t} \Fv - \BP_{t} \Fv \right\|_2 =0,
\end{equation}
 which gives the result in \eqref{equ:strong-convergence-Riccati-Aprox}.
\end{proof}

\begin{lemma}[Appendix \ref{sec:Proofs for  Graphon-LQR}] \label{lem:Convergnece of Approximated Riccati Solution} 
	Let $\BAP$ be generated by step function approximation from $\BP$ via the $N$ uniform partition of $[0,1]$ according to \eqref{equ:approximate-Riccati-Solution}. If  $\BSP$  converges strongly to the solution $\BP$  in $C_s([0,T]; \Sigma(L^2[0,1]))$, then for any $\Fx \in L^2[0,1]$, 
 \begin{equation*}
   \lim_{N\rightarrow \infty}\sup_{t\in[0,T]}\left\|\BAPt{t} \Fx- \BSPt{t} \Fx \right\|_2 =0.
 \end{equation*}
\end{lemma}

Let $Ricc(\BA, \BB, \BC, \BPZ)$ denote the  Riccati equation 
in \eqref{equ: Riccati-Equation} with initial condition $\BPZ$.

\begin{assumption}\label{ass:strong-convergence-in-Riccati}
~
	\begin{enumerate}
		\item For any $ N \geq 1$, 
$\BSPZ\in \mathcal{L}(L^2[0,1])$ is Hermitian and non-negative, and $\BSC\in \mathcal{L}(L^2[0,1])$.

 \item The system sequence $\{(\BSA,\BSB \}$ converges to $(\BA;\BB)$ as in \textup{Definition \ref{def:system-convergence}}.

 \item The sequences  $\{\BSC\}$ and $\{\BSPZ\}$ converge strongly to  $\BC$ and 
 $\BPZ$, respectively, as $N\rightarrow \infty$.
 
 \item $\BC$ and $\BSC$ are self-adjoint linear operators.  
	\end{enumerate}
	
\end{assumption}
 
\begin{theorem}
	
Consider a sequence of network systems $\{(\alpha_N I + A_N;\beta_N I + B_N)\}_{N=1}^{\infty}$ with $\{(\BSA; \BSB) \in \mathcal{G}^1_\mathcal{{AI}} \times \mathcal{G}^1_\mathcal{{AI}}  \}_{N=1}^{\infty}$  as the equivalent representation. 
	Let $\BP$ and $\BSP$ be the solutions to $Ricc(\BA, \BB, \BC, \BPZ)$ and $Ricc(\BSA, \BSB, \BSC, \BSPZ)$ respectively.  If \textup{Assumption \ref{ass:strong-convergence-in-Riccati}} holds, then
  	for any horizon $[0,T]$, $T>0$,
  			\begin{equation*}
          \lim_{N\rightarrow \infty} \BSP =\BP \quad \text{ in } C_s([0, T]; \Sigma(L^2[0,1])).
        \end{equation*}
\end{theorem}	
\begin{proof}
 From Theorem \ref{theorem: convergence in operator}, we know 
 for all $T>0$ and all $\Fx\in L^2[0,1]$,
  	$\lim_{N\rightarrow \infty}e^{t \BSA}\Fx= e^{t \BA}\Fx$ uniformly in $[0,T]$. 
	 Since the system sequence $\{(\BSA,\BSB) \}$ converges to $(\BA;\BB)$ as in Definition \ref{def:system-convergence}, $\{\BSB\}$ converges to $\BB$ in the strong operator sense.
We can now apply \cite[Theorem 2.2, Part IV]{bensoussan2007representation}, specialized to the Hilbert space $L^2[0,1]$. Since its hypotheses are then satisfied in the present case, the desired result follows. 
\end{proof}

%\subsubsection{Convergence of States and Convergence of Costs}

	Let  $\BSP$ denote the solution to the Riccati equation for $(\BSA; \BSB)$ that converges strongly to the solution $\BP$ of the Riccati equation for $(\BA;\BB)$. Let $\BAP$ be the step function approximation of $\BP$  generated  via the $N$ uniform partition of $[0,1]$ according to \eqref{equ:approximate-Riccati-Solution}.
	\begin{theorem}[Appendix \ref{sec:Proofs for  Graphon-LQR}] \label{thm: Convergence of the States} % 
	Consider the time horizon $[0, T]$.  Assume the sequence of initial conditions $\{\SxZ \in L^2[0,1]\}$ is convergent and \textup{Assumption \ref{ass:strong-convergence-in-Riccati}} holds.
	Let the optimal linear quadratic control law for $(\BSA; \BSB)$  be generated by 
  \begin{equation}\label{equ:optimal-control}
    \Fu^{N*}_\Ft= -\BSB^\TRANS\BSPt{(T-t)}\Fx^{N*}_\Ft,
  \end{equation}
	where the optimal state trajectory is given by $\Fx^{N*}$,
	and let the graphon approximate control law for $(\BSA; \BSB)$ be given by 
\begin{equation}
  \Sut= -\BSB^\TRANS\BAPt{(T-t)}\Sxt, 
\end{equation}
 where the corresponding state trajectory is given by $\Sx$.
	Then 
	$$
		\forall t \in [0, T], \quad \lim_{N \rightarrow \infty} \left\|\Fx^{N*}_\Ft - \Sxt\right\|_2 =0,
	$$
	and 
   $ \lim_{N\rightarrow \infty} \left|J(\Fu^{N*}) - J(\Su)\right|=0.$
	\end{theorem}

\section{Numerical Examples}
\subsection{Convergent Network Sequences with Sampled Weightings} \label{subsec:network-sampling}
The generation of a randomly sampled network of size $N$ from a graphon $\FA$ is specified as follows:
\begin{enumerate}[1)]
	\item Sample $N$  points from a uniform distribution in  $[0,1]$. 
Sort the sample points in the decreasing order of their values and label them from node $1$ to node $N$. 
 Denote  the node set by $V_N$ and the value of node $i \in V_N$ by $v_i$. 
 \item Connect the nodes $i,j \in V_N $   with edge weight $\FA(v_i,v_j)$  to generate the network $G_N$.  Then $A_{Nij} = \FA(v_i,v_j)$ is the $ij^{th}$ element of the adjacency matrix of $G_N$.
\end{enumerate}

If $\FA$ is almost everywhere continuous, then the step function $\SA$ of $A_N =[A_{Nij}]$ converges to $\FA$ in the $\delta_1$ metric as $N\rightarrow \infty$ (see e.g. \cite{borgs2011limits}), that is, $\delta_1({\SA, \FA}) \rightarrow 0,  \text{ as } N \rightarrow \infty.$
% %
%
% %
 Furthermore, this implies
$ 
\delta_2(\SA, \FA)  \rightarrow 0, \text{ as } N\rightarrow \infty$  since $\FA \in \ESO$ is uniformly bounded. By the generation procedure, we obtain the labeling that approximates the minimum distance between the network and the limit, and hence the sequence of networks converge in the $L^2[0,1]^2$ metric to the limit.% $\FA$.

As an example, we consider the following sinusoidal graphon $\FA$: for all $x,y \in [0,1],$ $$\FA(x,y)= 0.5\cos(2\pi(x-y))+0.25\cos(4\pi(x-y)). $$ The normalized eigenfunctions are  $ \Ff_1= \sqrt{2} \cos 2\pi(\cdot) $,  $\Ff_2= \sqrt{2} \cos 4\pi(\cdot),$
  $\Ff_3= \sqrt{2} \sin 2\pi(\cdot) $ and $\Ff_4= \sqrt{2} \sin 4\pi(\cdot)$
with  eigenvalues $\lambda_1= \lambda_3 = \frac{1}{4}$ and $\lambda_2= \lambda_4 =\frac{1}{8}.$
\subsection{Minimum Energy Graphon State-to-state Control}
Consider a network system evolving according to node averaging dynamics with $G_N$  describing the dynamic interactions. Suppose each node has an independent input channel. 
Denote the system by $(A_N; I_N)$, where  $A_N$ is the adjacency matrix of $G_N$ and $I_N$ is the identity input mapping.  
The network system $(A_N; I_N)$ with node averaging dynamics is therefore described by 
\begin{equation}
  \dot{x}^i_t=\frac{1}{N}\sum_{j=1}^{N} A_N{_{ij}}x^j_t+u^i_t, \hspace{0.2cm} x^i_t, u^i_t \in R,  i\in\{1,...,N\},
\end{equation} 
where $A_N{_{ij}}$ is sampled from the sinusoidal graphon. 

We solve the minimum energy control problem of driving the states of the network system $(A_N; I_N)$ to a terminal state $x^N_T$ from the origin over the time horizon $[0,T]$ with $T=2$.  
Here we consider the limit target terminal state
$
\Fx_\FT= \frac{1}{\sqrt{2\pi}} e^{-50 (\alpha-0.5)^2}, \alpha \in[0,1].
$

Based on Proposition \ref{prop:inverse of the controllability Gramian operator}, the system $(\FA; \BI)$ is exactly controllable and the inverse of the controllability Gramian operator is explicitly given by \eqref{equ:gramian-inverse}.
 The minimum control law based on \eqref{equ:mm-optimal-control-law} %
 is explicitly given by
 \begin{equation*} %\label{equ: minimum-energy-limit-control-law}
  \begin{aligned}
  \Fu_\Ft& =  \frac{1}{T} \Fx_\FT \\
  &+ \frac{1}{T}\sum_{\ell=1}^4 {\Ff_\ell} \langle \Fx_\FT, \Ff_\ell \rangle  \Big[-\frac{2 \lambda_\ell \left(\frac{e^{2 \lambda_\ell
       T}-1}{2 \lambda_\ell}-T\right) \Big(e^{\lambda_\ell (T-t)}-1\Big)}{e^{2 \lambda_\ell
       T}-1}\\
       & + \left(e^{\lambda_\ell (T-t)} -1\right) -\frac{2 \lambda_\ell \left(\frac{e^{2 \lambda_\ell T}-1}{2
       \lambda_\ell}-T\right)}{e^{2 \lambda_\ell T}-1}\Big],~~  t\in[0, T],
  \end{aligned}
\end{equation*}
with 
  $\langle \Fx_\FT, \Ff_1 \rangle = -0.116, ~ \langle \Fx_\FT, \Ff_2 \rangle = 0.064,~
    \langle \Fx_\FT, \Ff_3 \rangle =\langle \Fx_\FT, \Ff_4 \rangle = 0.$
Then the control law $u^N_{(\cdot)}$ for a network system $(A_N; I_N)$ generated based on the  approximation in \eqref{equ:mm-control-approximation}.
The error $\|{\Fx_\FT(\Fu)}-{\Fx^\FN_\FT(\Fu^{[\FN]})}\|_2$ is bounded as in Theorem \ref{thm:mm-graphon-main} and converges to 0 as $N\rightarrow \infty$. The result of a simulation for a network system with 100 nodes using the GSSC strategy is shown in Figure \ref{fig:Minimum-Energy-State-to-state-Control-Example}.
\begin{figure}[h]
\centering
  \subfloat[A network of size 100 in a sequence which converges to the sinusoidal graphon limit $\FA$]
  {\includegraphics[width=2.3cm]{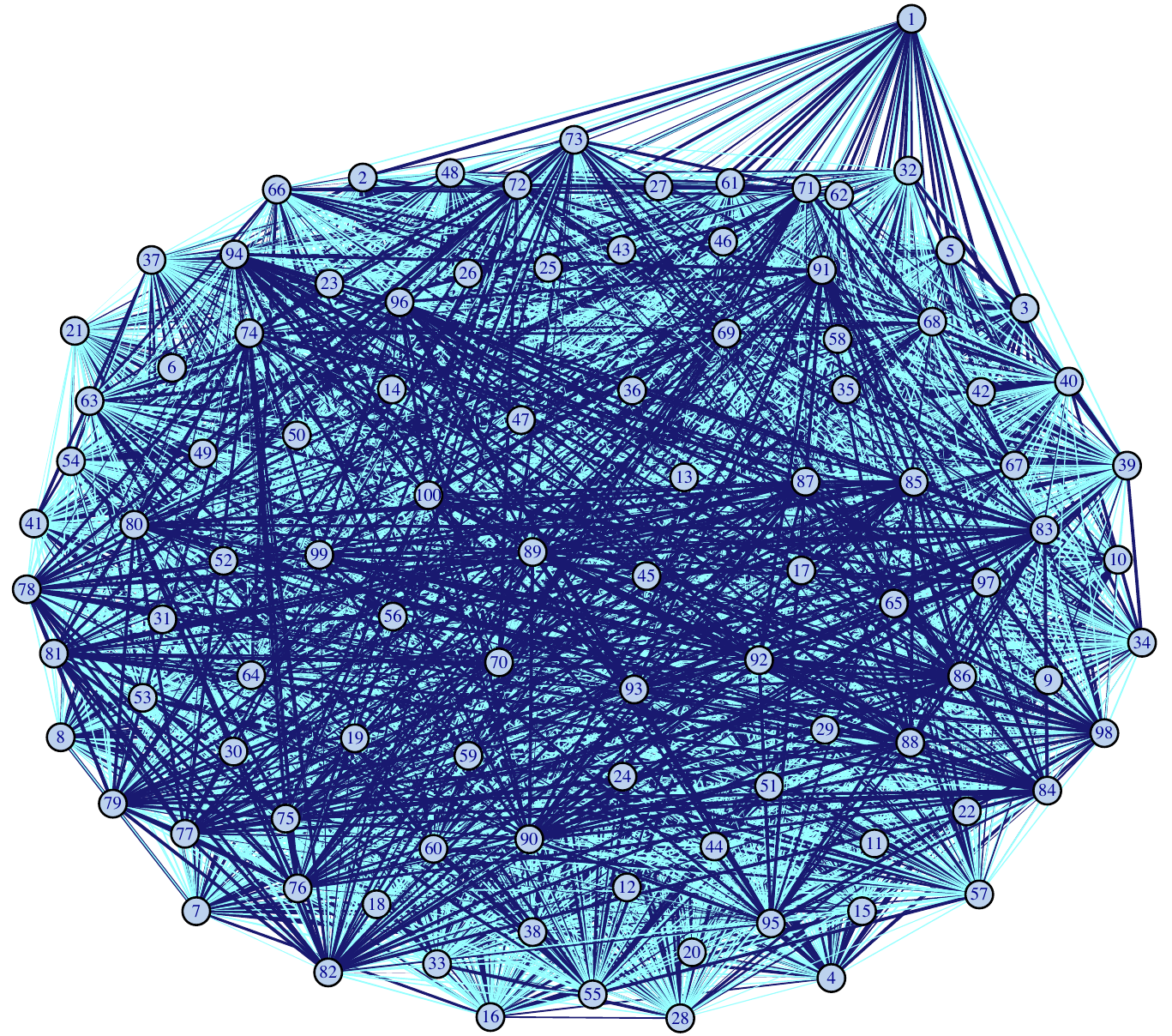}\qquad \includegraphics[width=2.5cm]{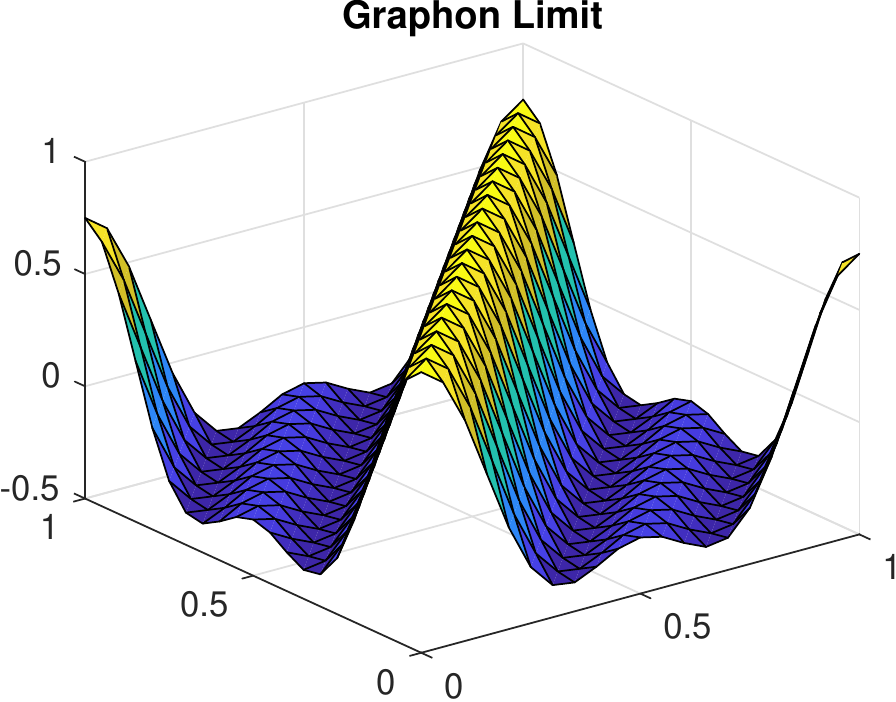}}\\
\subfloat[Target state, achieved terminal state and terminal state error]{\includegraphics[width=2.5cm]{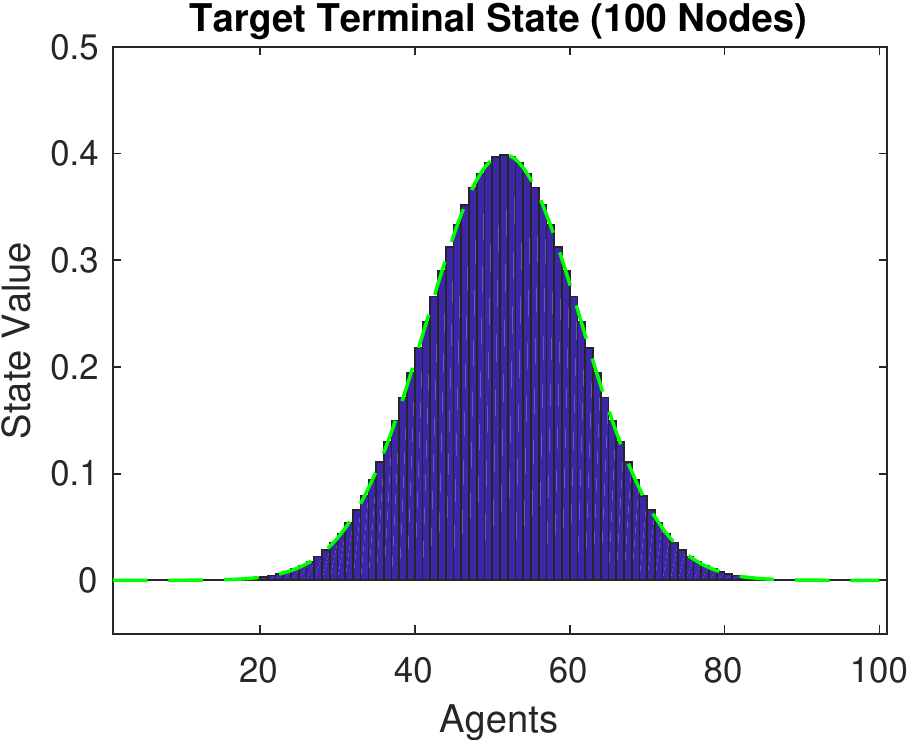}\includegraphics[width=2.5cm]{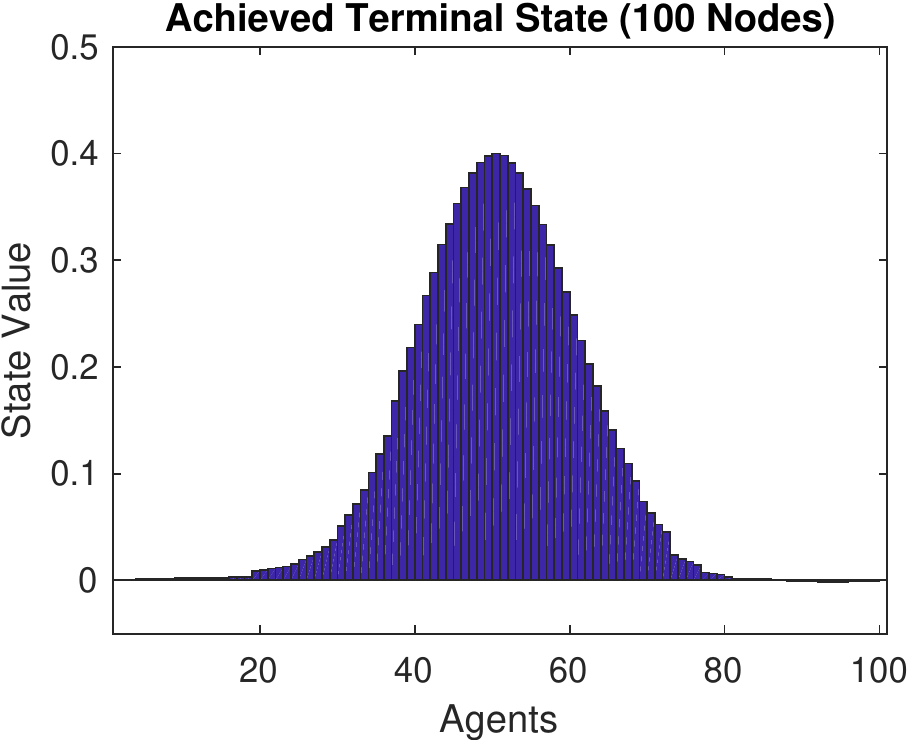}\includegraphics[width=2.5cm]{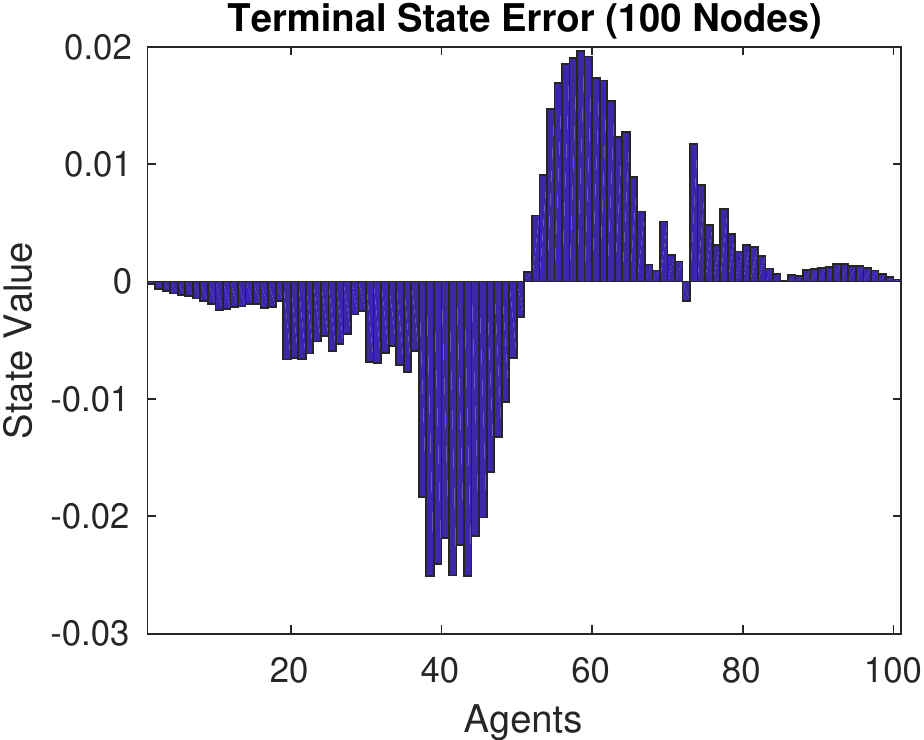}} 
\\
	\subfloat[State trajectory, control signal and its deviation from the optimal control signal]
	{\includegraphics[width=2.5cm]{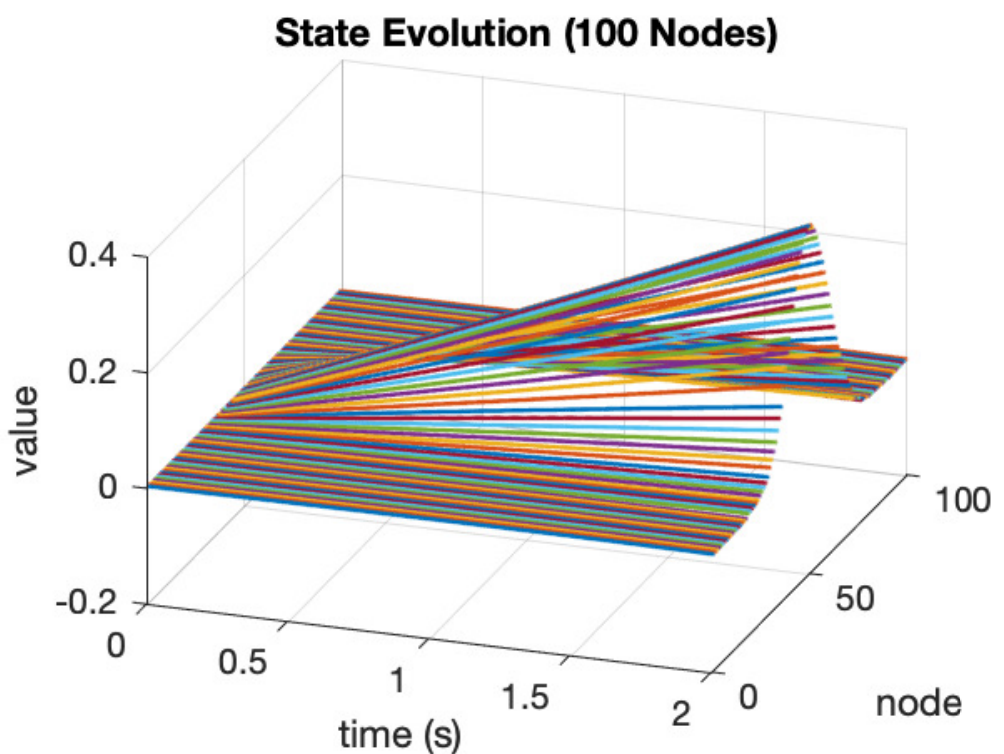} \includegraphics[width=2.5cm]{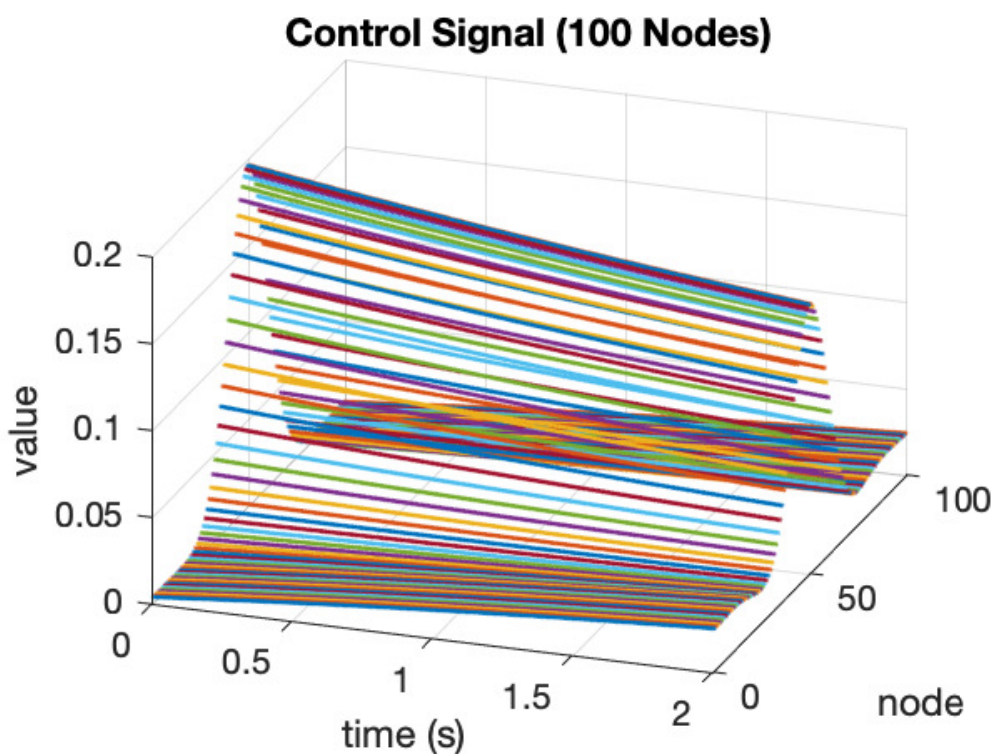}\includegraphics[width=2.5cm]{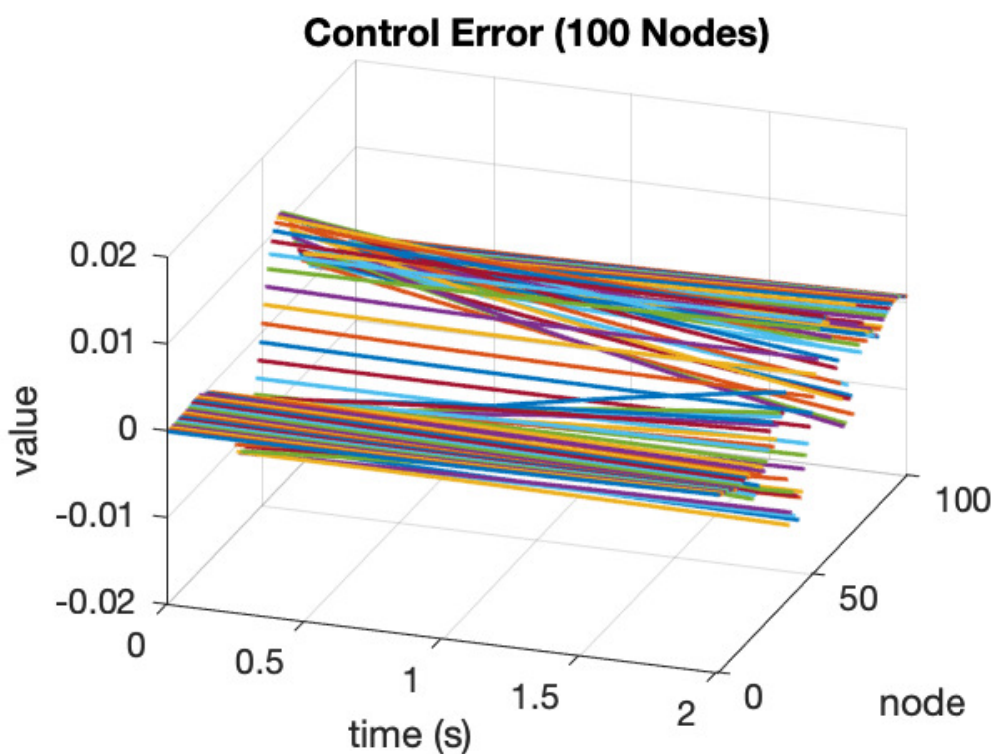}} 
	\caption{Minimum energy graphon state-to-state control} \label{fig:Minimum-Energy-State-to-state-Control-Example}
\end{figure}

\subsection{Graphon-Network LQR}
Consider a network system with dynamics given by 
\begin{equation}
  \dot{x}_t^i = \alpha x_t^i + \frac{1}{N}\sum_{j=1}^{N} A_{N_{ij}}x_t^j + \beta u_t^i,\quad x_t^i, u_t^i \in \BR .
\end{equation}
The objective is to design a control law that minimizes the following cost with network coupling:
\begin{multline}
  J(u) = \frac1N\sum_{i=1}^N \Big[ \int_0^T\Big(q(x_t^i - \frac{\eta}{N}\sum_{j=1}^N A_{N_{ij}} x^j_t)^2 + (u_t^i)^2\Big) dt \\
  + q_T\Big(x_T^i - \frac{\eta}{N}\sum_{j=1}^N A_{N_{ij}} x^j_T\Big)^2\Big],
\end{multline}
where $q,q_T\geq 0$.
That is we want to regulate the state of each subsystem to be close to the local weighted network average with small control effort. 
The equivalent formulation of this problem for the graphon system following \eqref{equ:step-function-dynamical-system} is given by 
\begin{equation}
\begin{aligned}
  &\DSxt = \alpha \Sxt + \SA \Sxt + \beta \Sut\\
  &J(\Su) = \int_0^T \Big[q \left\|(\BI - \eta \SA)\Sxt\right\|_2^2 + \left\|\Sut\right\|_2^2 \Big]dt \\
  & \qquad \qquad \qquad \qquad\qquad + q_T \left\|(\BI - \eta \SA)\SxT\right\|_2^2,
\end{aligned}
\end{equation}
where $ \Sxt, \Sut \in L_{pwc}^2[0,1], \SA \in \ESO$. 
The limit problem (if exists) or the approximate problem is given by 
\begin{equation}
\begin{aligned}
  &\dot{\Fx}_\Ft = \alpha \Fx_\Ft + \FA \Fx_\Ft + \beta \Fu_\Ft\\
  &J(\Fu) = \int_0^T \Big[q \left\|(\BI - \eta \FA)\Fx_\Ft\right\|_2^2 + \left\|\Fu_\Ft\right\|_2^2 \Big]dt \\
  & \qquad  \qquad \qquad \qquad\qquad + q_T \left\|(\BI - \eta \FA)\Fx_\FT\right\|_2^2,
\end{aligned}
\end{equation}
where $\FA \in \ESO$, $\Fx_\Ft, \Fu_\Ft \in L^2[0,1]$.
Let us consider a special case where $\FA$ permits an exact finite spectral decomposition as follows:
\begin{equation} 
  \FA(x,y) = \sum_{\ell=1}^{d} \lambda_\ell \Ff_\ell(x) \Ff_\ell(y), \quad (x,y)\in[0,1]^2,
\end{equation}
where $\lambda_1, \lambda_2,...,\lambda_d$ are non-zero eigenvalues of $\FA$ and $\Ff_1, \Ff_2,...,\Ff_d$ represent orthonormal eigenfunctions. 
Then the solution $\BP$ to the Riccati equation
 \begin{equation}
 \begin{aligned}
      & \dot{\BP}_t = (\alpha\BI+ \FA)^\TRANS \BP_t +  \BP_t(\alpha\BI+ \FA) - \beta^2 (\BP_t)^2 + q(\BI-\eta \FA)^2 \\
   & \BP_0= q_T(\BI-\eta \FA)^2
 \end{aligned}
 \end{equation}
 is given by 
$
  \BP_t = \breve \Pi_{t} \BI + \sum_{\ell=1}^d\left(\Pi_t^\ell - \breve \Pi_t\right)\Ff_\ell \Ff_\ell^\TRANS, 
$
where $\breve \Pi$ and $\Pi^\ell$ are the solutions to the following scalar Riccati equations
\begin{equation}
\begin{aligned}
  &\dot{ \breve \Pi}_t = 2 \alpha \breve \Pi_t -\beta^2  (\breve \Pi_t)^2 + q, \\
  & \dot{\Pi}^\ell_t = 2 (\alpha + \lambda_\ell)  \Pi_t^\ell -\beta^2  ( \Pi_t^\ell)^2 + q(1-\eta\lambda_\ell)^2, \\
  & \breve \Pi_0 = q_T,\quad  \Pi_0^\ell = q_T(1-\eta\lambda_\ell)^2,\quad  1 \leq \ell \leq d.
\end{aligned}
\end{equation}
The optimal control  for the limit problem is then given by 
\begin{equation*}
\begin{aligned}
   \Fu_\Ft  %
          & = - \beta \breve \Pi_{(T-t)} \Fx_\Ft - \beta \sum_{\ell =1}^d (\Pi_{(T-t)}^\ell-\breve \Pi_{(T-t)})  \langle\Fx_\Ft,\Ff_\ell\rangle \Ff_\ell.  \\
\end{aligned}
\end{equation*}
See \cite{ShuangAdityaCDC19,ShuangPeterCDC19W2} for the details of the solution method, which provides solutions to a more general class of graphon control problems with network couplings in states, controls and costs.

{If $\SA \rightarrow \FA$ as $N\rightarrow \infty$ in $L^2[0,1]^2$,} then all the conditions in Assumption \ref{ass:strong-convergence-in-Riccati} are satisfied. Hence one can generate approximate control for the original network system. 

Consider the following parameters: $\alpha =2$, $\beta=1.5$, $q=3$, $q_T=7$, $\eta=3$, $ n=1$ and $N=100$.
The numerical example is shown in Figure \ref{fig:graphon-lqr}. 

A direct solution to the original $N$-dimensional network LQR problem involves solving an $N\times N$ dimensional Riccati equation. However, the graphon approximate control method here involves only solving $d+1$ scalar Riccati equations, where $d$ is the number of non-zero eigenvalues of the graphon limit $\FA$. If the network is extremely large in size and the limit $\FA$ permits simple spectral representations, then the graphon control method would significantly reduce the computation complexity.

\begin{figure}[!t]
\centering
  \subfloat[This is a network of size 100 in a sequence that converges to the graphon $\FA$. This figure illustrates the structure, spectral properties and the step function representation.]{\includegraphics[width=2cm]{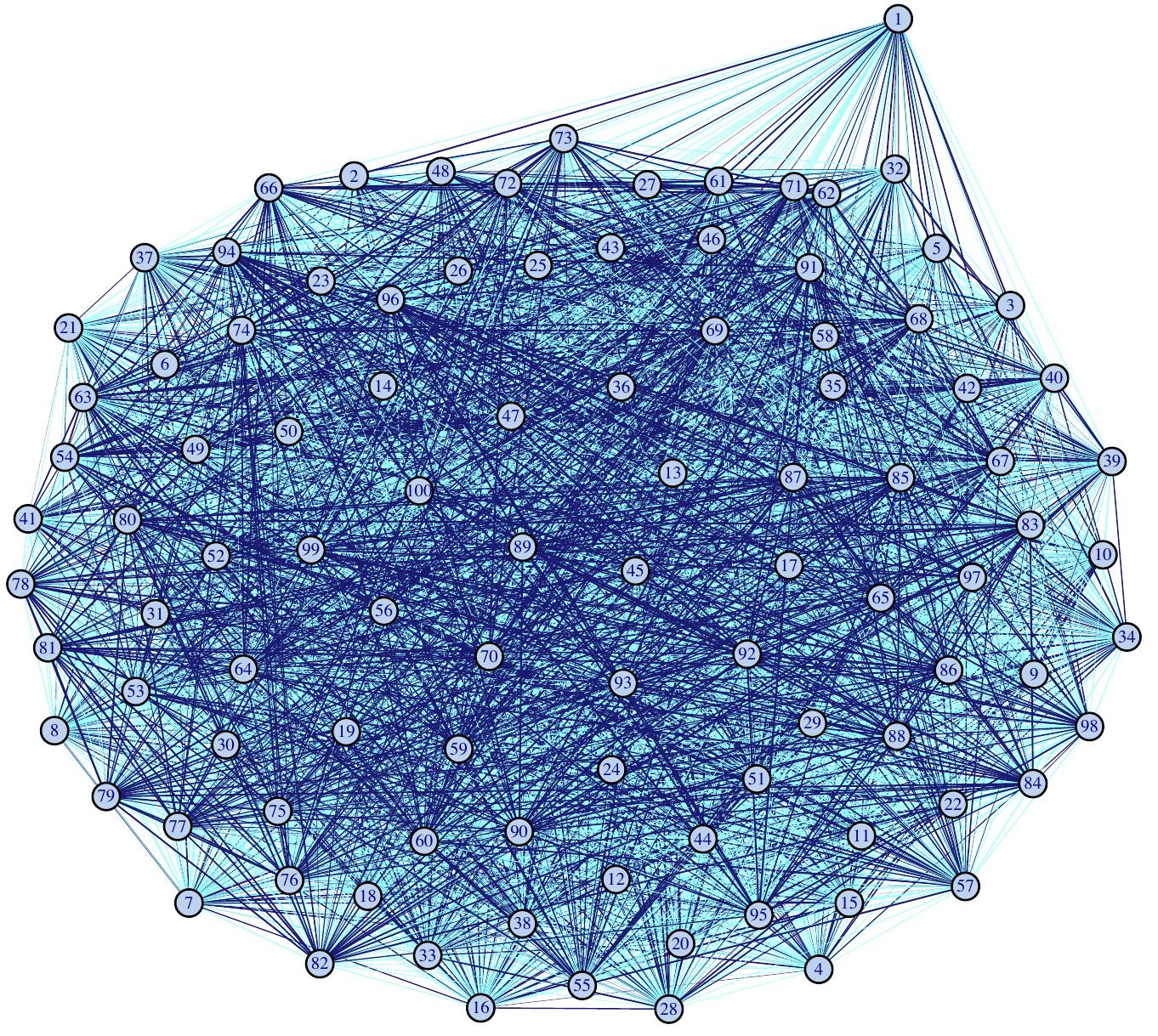}\includegraphics[width=6cm]{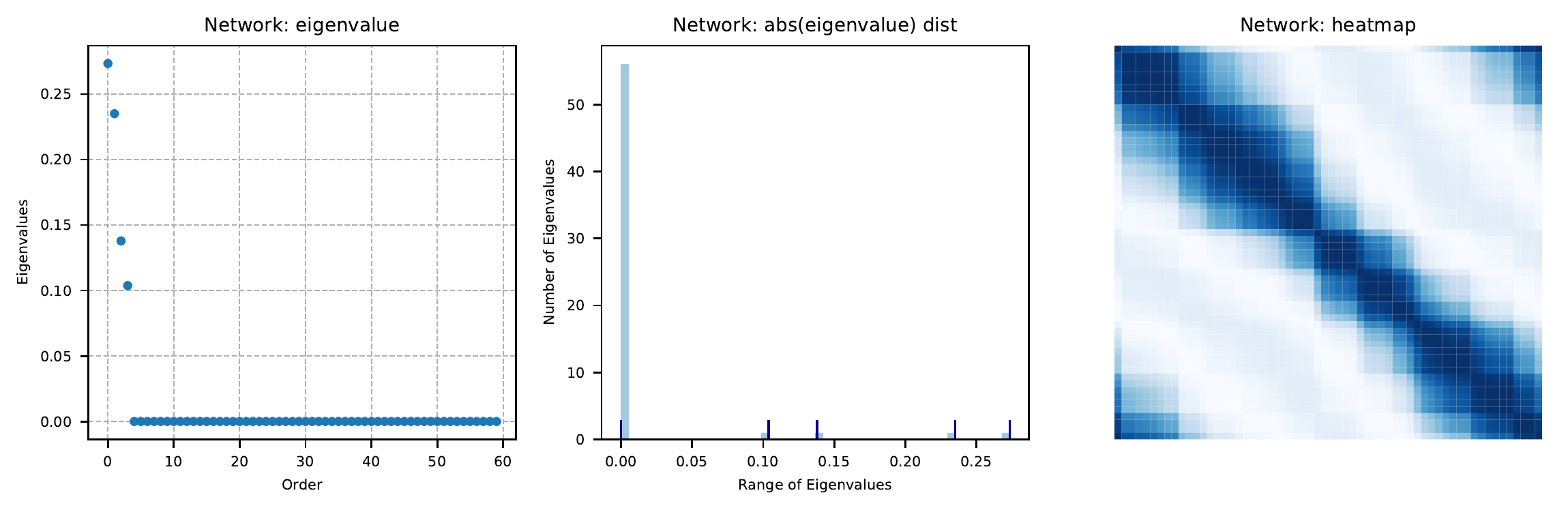}}\\
  \subfloat[Graphon approximate control and optimal control]{\includegraphics[width=8cm]{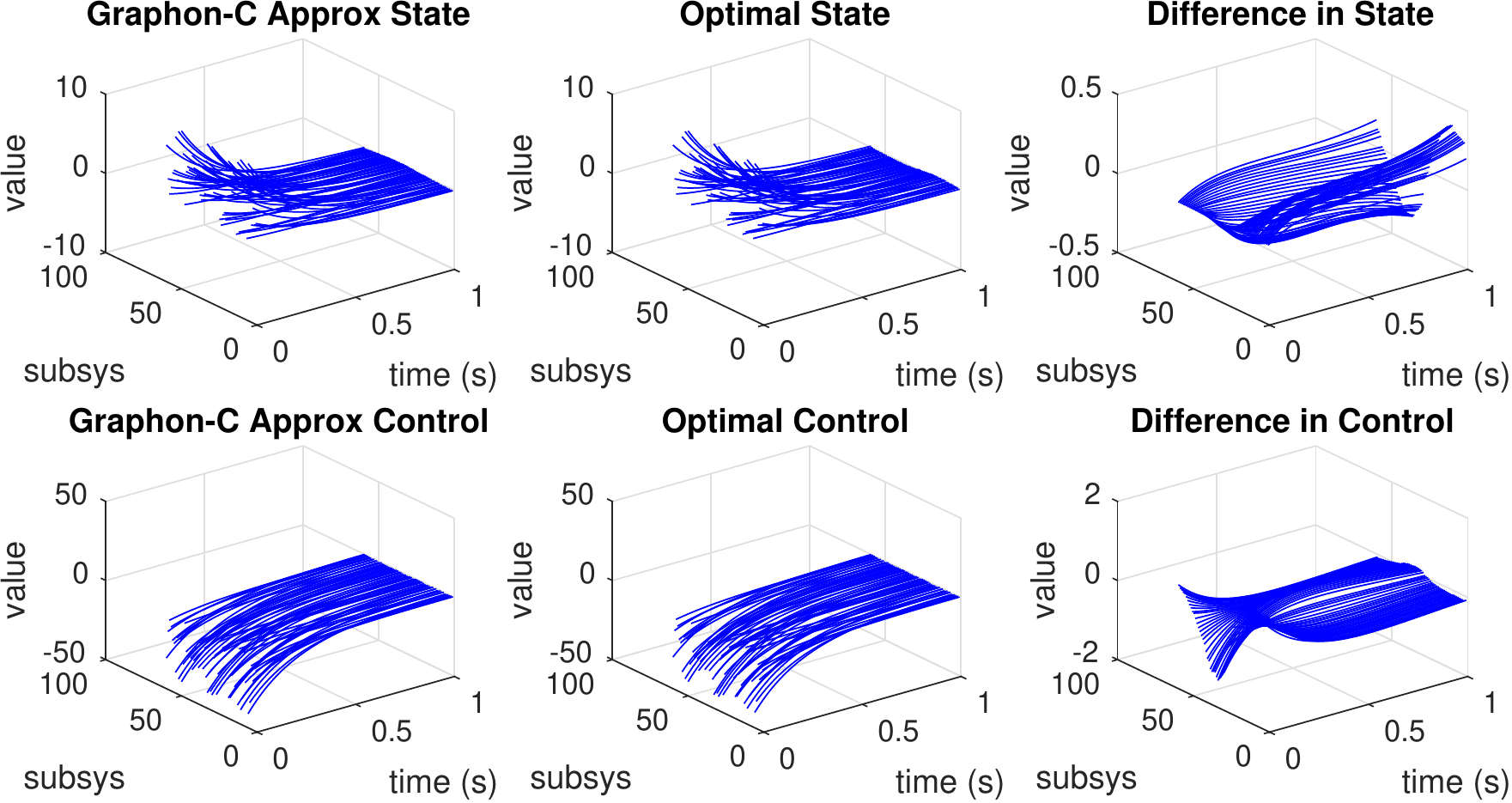}}
  \caption{Graphon approximate control and optimal control are applied to a network of size 100 in the sequence converging to the graphon limit. 
  With the graph interpreted as an $L^2[0,1]^2$ function, the distance between the graph and the graphon limit is  $0.061$ in the operator norm  and is $0.109$ in the $L^2[0,1]^2$ norm. The maximum trajectory difference from the optimal control is less than $3.320\%$ of the maximum of initial states. The graphon approximate control cost is $0.087\%$ higher than the optimal LQR control cost. }
  \label{fig:graphon-lqr}
\end{figure}

%  *-----------------  Performance  ----------------------------*
%  Op-norm diff (M-AppM) | L^2-norm (M-AppM) 
%  0.061       0.109 
 
%  Optimal Cost (JRicc) | Graphon Cost | Cost Diff 
%  5995.904 (6055.128)   6001.121     0.087%  

%  MaxState | MaxTrajectoryDiff | PercentDiff 
%  9.962     0.331      3.320%  
% *------------------------------------------------------------*

\section{Discussion}
The basic assumptions justifying the application of graphon control strategies are, first, that a given sequence of finite network systems converges to a unique limit graphon system (as in Definition \ref{def:system-convergence}) or that a given instance can be closely approximated by a graphon system, along with the  measure preserving bijections that achieve the best fit,  and second, that the corresponding control problem for the (limit) graphon system is tractable.
Under these assumptions, Theorems \ref{thm:mm-graphon-main} and \ref{thm: Convergence of the States} guarantee the effectiveness  graphon control strategies for the finite large-scale complex network systems.

A plausible empirical approach to model the required infinite limit graphon $G_{\infty}$ is to fit two dimensional Fourier series to the step function representation of the adjacency matrix. Such parametric modelling of empirical data could resemble parametric estimation in statistics and system identification.  Moreover, due to the compactness of graphon operators, representations or approximations by simple spectral decomposition are possible \cite{ShuangPeterCDC19W1} and will be analyzed in future work.

The generation of the graphon approximation models inevitably deals with relabellings. 
Although in the graphon control design methodology we do not restrict the labeling to be that of the best fit to the data, the control error still depends on the labeling of the nodes. 
Furthermore, the labeling of the nodes on the networks  is necessary for control implementation. 
 To find the best labellings for general graphs can be  a complex combinatorial task. 
Consequently, we underline that it is assumed in this paper that the best labeling is known beforehand, either through a specific way of growing the networks with labels that ensure the best fit to the limit  or through graphon estimation methods \cite{airoldi2013stochastic}.  % 

\section{Conclusion}

We propose a method to approximately control networks of linear systems using the inherent limit described by graphons.  Important aspects requiring further investigations include: (1) the application of the proposed limit graphon control strategy to asymmetric network systems where the interactions of dynamics are described by directed networks; %%
(2) the creation of an equivalent theory for sparse networks to the dense case developed here; (3)  the generation of a methodology for systematically fitting bivariate  analytic models to network data; (4) the application of graphon control to stochastic linear quadratic Gaussian problems;  (5) the analysis of decentralized graphon control via Mean Field Game theory \cite{PeterMinyiCDC18GMFG}; (6) the graphon control analysis to problems with non-symmetric local dynamic such as harmonic oscillator dynamics \cite{ShuangPeterTAC19W2}.

\appendices 
\section{Lemmas \ref{lem: stepfunction-operating-on-L2-functions}-\ref{lem:graphon-exp-norm}}
\begin{lemma} \label{lem: stepfunction-operating-on-L2-functions}
	Consider a step function $\SA \in \ESO$ defined via a partition $P=\{P_1,...,P_{nN}\}$ and $\Sutau \in L^2_{pwc}[0,1] $  defined via the same partition $P$ by 
	$$
		 \Sutau{(\alpha)}=nN\int_{P_i}{\Fu_\tau}(\beta)d\beta , \quad \forall \alpha \in P_i,
	$$
where $\Fu_\tau \in L^2[0,1]$.	
	Then the following result holds:
	\begin{align} \label{equ:step-function-projection}
   (\SA)^k\Sutau &=(\SA)^k \Fu_{\tau}, \quad k \geq 1.
	\end{align}
\end{lemma}
\begin{proof}

  Let $A_{ij}^{[N]} = \SA(x,y)$, for all $ (x,y) \in (P_i, P_j).$
	Then for all $x\in P_i,$ 
	\begin{equation} \label{equ:AsU}
	\begin{aligned}
				 &{[\SA\Fu_{\tau}](x)}  = \int_0^1\SA(x,y)\Fu_{\tau}(y)dy\\
         & =  \sum_j\int_{P_j}A_{ij}^{[N]} \Fu_{\tau}(y)dy
				=\sum_j  A_{ij}^{[N]}\int_{P_j} \Fu_{\tau}(y)dy\\
				&= \sum_j A_{ij}^{[N]}\cdot \mu(P_j)\cdot \Fu_{\tau}^{[\FN]}(x),
%				%
	\end{aligned}
	\end{equation}
	\begin{equation}\label{equ: AsUN}
	\begin{aligned}
			&{[\SA\Fu_{\tau}^{[\FN]}](x)} = \sum_j\int_{P_j}A_{ij}^{[N]}\Fu_{\tau}^{[\FN]}(y)dy,\qquad \quad \\
%				%
				&= \sum_jA_{ij}^{[N]}\cdot \mu(P_j)\cdot \Fu_{\tau}^{[\FN]}(x).
	\end{aligned}
	\end{equation}
(\ref{equ:AsU}) and (\ref{equ: AsUN}) give the equality 
$\SA \Sutau =\SA \Fu_\tau$, which immediately implies \eqref{equ:step-function-projection}.
\end{proof}

\begin{lemma} \label{lem:operator-norm-and-L2-norm}
For any graphon $\FW$ or any function $\FW$ in $L^2[0,1]^2$, $\|\FW\|_{\textup{op}} \leq \|\FW\|_2.$
\end{lemma}
\begin{proof}
$$
  \begin{aligned}
    &\|\FW\|_{\text{op}}  = \sup_{x \neq 0, x \in L^2[0,1]}\frac{\|\FW x\|_2}{\|x\|_2}\\
          &  = \sup_{x \neq 0, x \in L^2[0,1]}\frac{\sqrt{\int_0^1 \left[\int_0^1 \FW(\alpha, \beta) x(\beta) d\beta\right]^2 d\alpha}}{\|x\|_2} \\
          & \leq \sup_{x \neq 0, x \in L^2[0,1]}\frac{\sqrt{ \int_0^1 \left[\int_0^1 \FW^2(\alpha, \beta) d\beta \int_0^1 x^2(\beta) d\beta \right] d\alpha}}{\|x\|_2}\\
          & = \sup_{x \neq 0, x \in L^2[0,1]}\frac{\|x\|_2   \sqrt{ \int_0^1 \int_0^1 \FW^2(\alpha, \beta) d\beta d\alpha } }{\|x\|_2}  
           = \|\FW\|_2.
  \end{aligned} 
$$     
\end{proof}

\begin{lemma}\label{lem:graphon-exp-norm}
   For any $\Fu\in L^2[0,1]$,  any $\FA \in \ESO$, and any $t\in[0,T]$, the following inequalities hold
   \begin{equation*} 
     \left\|e^{\FA t} \Fu \right\|_2 \leq e^{t\|\FA\|_\textup{op}} \|\Fu\|_2\leq e^{t\|\FA\|_2} \|\Fu\|_2.
   \end{equation*}
\end{lemma}
\begin{proof}
By recursively applying the definition of the operator norm,   we obtain that $\|\FW^k\|_{\text{op}}\leq \|\FW\|_{\text{op}}^k,~ k \in \BN \cup \{0\}$, for any $\FW\in \ESO$. Hence,
\begin{equation*}
  \begin{aligned}
    \left\|e^{\FA t} \Fu \right\|_2 &\leq  \sum_{k=0}^\infty \frac{1}{k!} t^k  \|\FA^k \Fu\|_2 
    \leq \sum_{k=0}^\infty \frac{1}{k!} t^k  \|\FA\|_{\textup{op}}^k\| \Fu\|_2 \\
    & = e^{t\|\FA\|_\textup{op}}\|\Fu\|_2\leq e^{t\|\FA\|_2} \|\Fu\|_2. \quad \text{(by Lemma \ref{lem:operator-norm-and-L2-norm})}
  \end{aligned}
\end{equation*}
\end{proof}
  
  Results in Lemma \ref{lem: stepfunction-operating-on-L2-functions}, %
   Lemma \ref{lem:operator-norm-and-L2-norm} and Lemma \ref{lem:graphon-exp-norm} generalize to functions in any uniformly bounded subsets of $\ES$, that is, any set of symmetric measurable functions $\FW:[0,1]^2\rightarrow I$ where $I$ is a bounded interval in $\BR$.
\section{Proofs of Graphon System Properties}\label{sec:Proofs for Graphon System Properties}
\subsection{Proof of Lemma \ref{lem: Network-Stepfunction-Graphon}}
\begin{proof}
Let $P^{nN}=\{P_1,..., P_{nN}\}$  be the uniform partition of $[0,1]$ with $P_i= [\frac{i-1}{nN}, \frac{i}{nN}), 1\leq i <nN,$ and $P_{nN}=[\frac{nN-1}{nN}, 1]$. Consider any $\Fx_\Fs \in L^2_{pwc}[0,1]$ and its corresponding vector $x\in \BR^{nN}$ following the vetor-to-PWC-function mapping $M_G$.
Since
\begin{equation*}
	[\SA \Fx_\Fs](\alpha)=\int_0^1\SA(\alpha, \beta) \Fx_\Fs(\beta) d\beta , \quad \Fx_\Fs\in L^2_{pwc}[0,1],
\end{equation*}
it follows that for all  $\alpha \in P_i$,
\begin{equation} \label{equ: step function integration and circle operation equivalance}
\begin{aligned}
	 {[\SA \Fx_\Fs]}(\alpha)%
	&=\sum_{j=1}^{nN} \int_{P_j} \SA(\alpha, \beta) \Fx_\Fs(\beta) d\beta\\
	&=\sum_{j=1}^{nN} \int_{P_j} A_{Nij} x_j d\beta
	=\sum_{j=1}^{nN} \frac{1}{nN} A_{Nij} x_j \\
	&= \frac{1}{nN} [A_Nx]_i=[A_N\circ x]_i,
	\end{aligned}
\end{equation}
where $x_j$ denotes the $j^{th}$ element of $x \in R^{nN}$ and $[A_Nx]_i$ denotes the $i^{th}$ element of $A_Nx \in R^{nN}$.
This implies that the step function $\SA$ in the graphon  space, considered as an  operator, represents a mapping in $L^2[0,1]$; this operator  is equivalent to the matrix transformation $A_N$ with $\circ$ operation in $R^{nN}$ and  the corresponding  mapping $M_G$. A similar conclusion holds for $\SB$ and $B_N$.
Furthermore, it is obvious that 
\begin{equation}
 \forall \gamma \in P_i, \quad \alpha_N \BI \Fx_\Fs (\gamma) = \alpha_N \Fx_\Fs (\gamma) = \alpha_N x_i, 
\end{equation}
and a similar conclusion holds for $\beta_N \BI$ and $\beta_N I$. 
Hence we conclude that 
 the trajectory of the system $(\alpha_N I + A_N;\beta_N I+ B_N)$ corresponds one-to-one to that of $(\alpha_N \BI + \SA;\beta_N \BI + \SB)$ under the corresponding vetor-to-PWC-function mapping $M_G$. 
\end{proof}

\subsection{Proof of Theorem \ref{theorem: convergence in operator}}
\begin{proof}
Let us define
$
P_k(x,y)=\sum_{i=0}^k x^{k-i}y^{i}.
$
Then
$x^k-y^k=(x-y)P_{k-1}(x,y).$
 We obtain that for $k\geq1$
$$
	\FA^k_\mathbf{N}- \FA_*^k= P_{k-1}(\FA_\mathbf{N}, \FA_*)(\FA_\mathbf{N}- \FA_*).
$$
  Since $\FA^{(k-i-1)} \FA_*^i\in \ESO$, for all $i \in\{0,1,...k-1\},$
  we know that
    $\|\FA^{(k-i-1)}\FA_*^i\|_2\leq 1.$
Hence, by Lemma \ref{lem:operator-norm-and-L2-norm}, 
\begin{equation}
\begin{aligned}
  \|P_{k-1}(\FA_\mathbf{N}, \FA_*)\|_{\text{op}}%%
   &\leq \sum_{i=0}^{k-1}\|\FA^{(k-i-1)}\FA_\Fs^i\|_2  \leq k \cdot 1. 
\end{aligned}
\end{equation}
For an arbitrary $\Fx\in L^2[0,1]$ and finite $t$, $0 \leq t < \infty$,
\begin{equation}\label{equ:proof-exp-x}
	\begin{aligned}
		\big\|&e^{\FA_\mathbf{N}t}\Fx-e^{\FA_*t}\Fx \big\|_2 %
		\sum_{k=1}^{\infty}\frac{t^k}{k!}\|(\FA_\mathbf{N}^k - \FA_*^k) \Fx\|_2\\
		& \leq \sum_{k=1}^{\infty}\frac{t^k}{k!}\|P_{k-1}(\FA_\mathbf{N}, \FA_*)\|_{\text{op}} \cdot\|\FA_{\Delta}^{\mathbf{N}}\|_{\text{op}} \|\Fx\|_2\\
		& \leq \sum_{k=1}^{\infty}\frac{t^k}{k!}\cdot k \cdot\|\FA_{\Delta}\|_{\text{op}} \|\Fx\|_2
			=te^t\cdot \|\FA_{\Delta}^{\mathbf{N}}\|_\text{op}\|\Fx\|_2.
	\end{aligned}
\end{equation}
It follows that for  $t\in [0, T]$
\begin{equation}
	\begin{aligned}
		 \left\|e^{\FA_\mathbf{N}t}\Fx-e^{\FA_*t}\Fx\right\|_2
	%	%
		& \leq Te^T\cdot \|\FA_{\Delta}^\mathbf{N}\|_\text{op}\|\Fx\|_2.
	\end{aligned}
\end{equation}
Hence the convergence is point-wise in time and  uniform in $t$ over $[0, T]$ and hence  \eqref{equ:convergence-exp-1} holds.   Furthermore,
\begin{equation}\label{equ:proof-exp-x2}
\begin{aligned}
   &\left\|e^{(\alpha_N \BI + \FA_\FN)t}\Fx - e^{(\alpha \BI + \FA_*)t}\Fx\right\|_2 
   = \left\|e^{\alpha_N t} e^{\FA_\FN t} \Fx - e^{\alpha t}e^{\FA_* t}\Fx\right\|_2\\
   &\leq \left\|e^{\alpha_N t}(e^{\FA_\FN t}- e^{\FA_* t})  \Fx\right\|_2+ \left\| (e^{\alpha_N t}- e^{\alpha t})e^{\FA_* t}\Fx\right\|_2\\
   & \leq e^{\alpha_N t} te^{t} \|\FA_{\Delta}^\mathbf{N}\|_\text{op}\|\Fx\|_2 + |\alpha-\alpha_N| t  e^{(L_\alpha +\|\FA_*\|_\text{op})t}\left\|\Fx\right\|_2, ~
\end{aligned}
\end{equation}
where $L_\alpha= \max\{|\alpha|, |\alpha_N|\}$.
 The last step of \eqref{equ:proof-exp-x2} is due to equation \eqref{equ:proof-exp-x}, Lemma \ref{lem:graphon-exp-norm} and the following
    \begin{equation}
      \begin{aligned}
        |e^{\alpha t} &-  e^{\alpha_N t}| = \left|\sum_{k=0}^\infty \frac{1}{k!}[(\alpha )^k - (\alpha_N )^k]t^k\right|\\
        & \leq \sum_{k=1}^{\infty} \frac{1}{k!}|\alpha-\alpha_N|\cdot k L_\alpha^{(k-1)} t^k
         = |\alpha-\alpha_N| t e^{tL_\alpha }
      \end{aligned}
    \end{equation}
    where $t>0$ and $L_\alpha= \max\{|\alpha|, |\alpha_N|\}$.
An immediate implication of \eqref{equ:proof-exp-x2} is that for $t\in[0,T]$,
\begin{equation}\label{equ:proof-exp-x3}
\begin{aligned}
  &\left\|e^{(\alpha_N \BI + \FA_\FN)t}\Fx - e^{(\alpha \BI + \FA_*)t}\Fx\right\|_2  \\
     & ~ \leq e^{\alpha_N T} Te^{T} \|\FA_{\Delta}^\mathbf{N}\|_\text{op}\|\Fx\|_2 + |\alpha-\alpha_N| T  e^{(L_\alpha +\|\FA_*\|_{\text{op}})T}\left\|\Fx\right\|_2, ~\\
%  %
\end{aligned}
\end{equation}
By the convergence of $\{\alpha_N\}$, we know $\{\alpha_N\}$ and $\{L_\alpha\}$ are both uniformly bounded. 
This, together with \eqref{equ:proof-exp-x3}, implies the convergence in \eqref{equ:convergence-exp-2} which is uniform in time over a closed time horizon $[0,T]$. 
\end{proof}

\section{Proofs for Exact Controllability}\label{sec:Proofs for Exact Controllability}

\subsection{Proof of Theorem \ref{thm:sufficient-condition-for-exact-controllability}}
\begin{proof}
% %

Since any $\FA \in \ESO$ defines a self-adjoint and compact operator, it has a discrete spectrum \cite{lovasz2012large}, and  the maximum absolute value of eigenvalues of $\FA$ equals to the operator norm
\cite[Theorem 12.31]{rudin1991functional}, that is, 
  $\|\FA\|_{\text{op}} = \max_\ell |\lambda_\ell|$, where $\{\lambda_\ell\}$ denotes the set of eigenvalues of $\FA$.
	For any $\FA \in \ESO $, $\FA$ is a bounded operator on $L^2[0,1]$, that is,
   there exists some finite $c_1>0$, such that $ \|\FA\|_{\text{op}} \leq c_1$. Therefore for $\ell \in \BN$, $c_1 \geq \lambda_{\ell} \geq -c_1$ and hence for $t>0$,
		$\|e^{\FA t}\|_{\text{op}}\geq  e^{\lambda_\ell t}   \geq e^{-c_1 t}>0.$
  Hence based on Lemma \ref{lem:seperation-graphon-exp}, for $t>0$,
  \begin{equation}
    \|e^{\BA t}\|_{\text{op}}= \|e^{(\alpha\FI+\FA) t}\|_{\text{op}} = e^{\alpha t}\|e^{\FA t}\|_{\text{op}} \geq  e^{(\alpha-c_1) t}>0.
  \end{equation}
 This implies  $e^{\BA t}$ as an operator is uniformly positive definite.

 Since all the values in the spectrum of $\BB\BB^\TRANS$ as a self-joint operator are lower bounded by a positive constant, %%
 there exists $c>0$ such that, for all $x\in L^2[0,1]$,
 
 $\langle \BB \BB^\TRANS x,  x\rangle \geq c \|x\|_2^2.
 $ 
 See e.g. \cite[Theorem 12.12]{rudin1991functional}.
  Consider the time horizon $[0, T]$. For any $h\in L^2[0,1]$,
\begin{equation}
	\begin{aligned}
		\langle \BW_T h, h \rangle & =  \int_0^T \langle\BB \BB^\TRANS e^{\BA^\TRANS t} h,  e^{\BA^\TRANS t} h \rangle dt 
		\geq c \int_0^T \| e^{\BA^\TRANS t} h \|_2^2dt\\
	& \geq  cT (e^{(\alpha-c_1) T})^2 \|h\|_2^2 ,
	\end{aligned}
\end{equation}
and hence the system $(\BA;\BB)$ is exactly controllable.
\end{proof}
\subsection{Proof of Proposition \ref{prop:graphon-system-with-graphon-input-not-exact-controllable}}	
\begin{proof}
	By Lemma \ref{lem:operator-norm-and-L2-norm}, since $\FA$ and $\FB$ are graphons in $\ESO$, there exists $c_1\geq 0$ and $c_2\geq 0$, such that 
	\begin{equation}
		 c_1 \geq \|\FA\|_2\geq \|\FA\|_{\text{op}}  \quad \text{and}\quad  c_2 \geq \|\FB\|_2\geq \|\FB\|_{\text{op}}.
	\end{equation}
%		%
	Hence 
	\begin{equation}
		\begin{aligned}
				\|e^{\FA t}\|_{\text{op}} %= & \sup_{x\in L^2[0,1], \|x\|_2=1}{\|e^{\FA t} x\|_2} \\ 
        &  \leq \sup_{x\in L^2[0,1], \|x\|_2=1} \sum_{k=0}^\infty \frac1{k!}\|\FA t\|_{\text{op}}^k \|x\|_2 \\
				&=  e^{\|\FA\|_{\text{op}} t} \leq e^{c_1 t}, \quad t \in [0, T]. 
		\end{aligned}
	\end{equation}
Therefore 
	\begin{equation}
		\begin{aligned}
			\|\BW_T&\|_2   \leq \int_0^T \left\|e^{\FA t}\FB\FB^\TRANS e^{\FA^\TRANS t}\right\|_2 dt 
						 \leq \int_0^T\left\|e^{\FA t}\FB \right\|_2^2 dt\\
						& \leq \int_0^T\left(\|e^{\FA t}\|_{\text{op}} \|\FB\|_2\right)^2 dt
						 \leq T (e^{c_1 T} c_2)^2 < \infty,
		\end{aligned}
	\end{equation}
	which implies $\BW_T \in L^2[0,1]^2$ and hence $\BW_T$ is a compact (and self-joint) operator on $L^2[0,1]$ functions (see e.g. \cite[Chapter 2, Proposition 4.7]{conway2013course}). This means that $\BW_T$  has a countable number of nonzero (real) eigenvalues $\{\lambda_1, \lambda_2, ...\}$ such that $\lambda_n \rightarrow 0$, and each eigenvalue has finite multiplicity (see e.g. \cite{bensoussan2007representation}). Therefore $\BW_T$ is not uniformly positive definite and hence the system $(\FA;\FB)$ is not exactly controllable. 
\end{proof}
\section{Proofs for State-to-state Graphon Control}\label{sec:Proofs for State-to-state Graphon Control}

\begin{lemma}\label{lem:state-traject-identity-in}
  Consider any $\Fu\in L^2\left([0,T];L^2[0,1]\right)$,  $ \BA, ~\BSA \in  \mathcal{G}^1_\mathcal{{AI}}$, and any $t\in[0,T]$. Let $\BA = (\alpha \BI+ \FA)$, $\BSA= (\alpha_N \BI+ \SA)$. Then the following inequality holds
\begin{equation}
  \begin{aligned}
    &\left\|\int_0^T\left(e^{\BA(T-t)} -e^{\BSA(T-t)}\right)\Fu_\Ft dt \right\|_2\\
     \leq &\left|\alpha- \alpha_N\right| \int_0^T L_{\alpha}(T-t)e^{(L_{\alpha}+\|\FA\|_\textup{op})(T-t)}\left\|\Fu_\Ft\right\|_2 dt \\
      & \quad + \|\FA_{\Delta}^{\mathbf{N}}\|_\textup{op} \int_0^T e^{(\alpha_N+1)(T-t)} (T-t) \|\Fu_\Ft\|_2 dt
  \end{aligned}
\end{equation}
    where $L_{\alpha}= \max\{|\alpha|, |\alpha_N|\}$ and $\FA_{\Delta}^{\mathbf{N}}=\FA-\SA$. % %
\end{lemma}
\begin{proof}
An application of \eqref{equ:exp-of-A+I} in Theorem \ref{theorem: convergence in operator} leads to the above result.
\end{proof}
 \subsection{Proof of Theorem \ref{thm:mm-graphon-main}}
\begin{proof}
   \begin{equation}\label{equ:main-mm-terminal-proof}
   \begin{aligned}
     \|&{\Fx_\FT(\Fu)} - {\Fx^\FN_\FT(\Fu^{[\FN]})}\|_2 \\
      = &\left\|\int_0^T e^{\BA(T-t)}\BB \Fu_\Ft dt - \int_0^T e^{\BSA(T-t)}\BSB \Sut dt\right\|_2 \\
     \leq& \left\|\int_0^T [e^{\BA(T-t)}-e^{\BSA(T-t)}]\BB \Fu_\Ft dt\right\|_2\\
      + &\left\|\int_0^T e^{\BSA(T-t)}\left[\BB \Fu_\Ft -\BSB \Sut\right]dt \right\|_2\\
     \leq &\left|\alpha- \alpha_N\right| (|\beta| + \|\FB\|_\textup{op}) \\
     & \qquad \qquad  \int_0^T (T-t)e^{(L_\alpha+\|\FA\|_\textup{op})(T-t)}\left\|\Fu_\Ft\right\|_2 dt \\
     +  &\|\FA_{\Delta}^{\mathbf{N}}\|_\textup{op}(|\beta| + \|\FB\|_\textup{op}) \int_0^T e^{(\alpha_N+1)(T-t)} (T-t) \|\Fu_\Ft\|_2 dt\\
      +    &  \left\|\int_0^T e^{\BSA(T-t)}\left[\BB \Fu_\Ft -\BSB \Sut\right]dt \right\|_2 \quad \text{(by Lemma \ref{lem:state-traject-identity-in})}\\
     \leq &\left|\alpha- \alpha_N\right| (|\beta| + \|\FB\|_\textup{op}) \\
     & \qquad \qquad \int_0^T  (T-t)e^{(L_\alpha+\|\FA\|_\textup{op})(T-t)}\left\|\Fu_\Ft\right\|_2 dt \\
     +  &\|\FA_{\Delta}^{\mathbf{N}}\|_\textup{op}(|\beta| + \|\FB\|_\textup{op}) \int_0^T e^{(\alpha_N+1)(T-t)} (T-t) \|\Fu_\Ft\|_2 dt\\
     + & (|\beta-\beta_N|+ \|\FB-\SB\|_\textup{op})\\
     & \qquad \qquad  \int_0^T e^{\left(\alpha_N + \left\|\SA\right\|_\textup{op}\right)(T-t)} \|\Fu_\Ft\|_2 dt \\
     + &|\beta_N|\int_0^T e^{\left(\alpha_N + \|\SA\|_\textup{op}\right)(T-t)} \left\|\Fu_\Ft - \Sut\right\|_2 dt .
   \end{aligned}
   \end{equation}
   The last step is due to Lemma \ref{lem:exp-of-I+A} and the following 
   \begin{equation}
   \begin{aligned}
     & \left[\BB \Fu_\Ft - \BSB \Sut\right] = \left[(\beta \BI + \FB) \Fu_\Ft - (\beta_N \BI + \SB)\Sut\right]\\
     &=\left[(\beta\Fu- \beta_N\Sut )+ \FB \Fu_\Ft -\SB \Sut\right]\\
      &=\left[(\beta\Fu- \beta_N\Sut )+ \FB \Fu_\Ft -\SB \Fu_\Ft \right]~~\text{(by Lemma \ref{lem: stepfunction-operating-on-L2-functions})}\\
      & = (\beta- \beta_N) \Fu_\Ft + \beta_N (\Fu_\Ft-\Sut) + (\FB-\SB)\Fu_\Ft.
   \end{aligned}
   \end{equation}

By the convergence of $\{\alpha_N\}, \{\beta_N\}, \{\SA\}$ and $\{\SB\}$, we obtain that they are uniformly bounded. Based on Proposition \ref{prop:contraction-L2-function}, $\|\Sut\|_2\leq \|\Fu_\Ft\|_2$, for any $t\in [0,T]$ and any $N\in \BN$. Hence from \eqref{equ:main-mm-terminal-proof},  we obtain 
\begin{equation*}
  \lim_{N\rightarrow\infty}\left\|{\Fx_\FT(\Fu)} - {\Fx^\FN_\FT(\Fu^{[\FN]})}\right\|_2 =0.
\end{equation*}
\end{proof}
%
%
%
%%%%%%%
\section{Inverse of the Controllability Gramian}\label{sec:Inverse of the Controllability Gramian}
Let $\BT$ and $\BS$ be linear bounded operators on a Hilbert space.% 
\begin{proposition}[\cite{bensoussan2007representation}] \label{prop: inverse operator exists}
Assume that $\BT$ and $\BS$ are symmetric and nonnegative. Then $\BI+\BT\BS$ is one-to-one and onto; moreover
  $\|\BS(\BI+\BT\BS)^{-1}\|\leq\|\BS\|,$
and
  $\|(\BI+\BT\BS)^{-1}\| \leq 1+\|\BT\| \|\BS\|.$
\end{proposition}

 Following this we prove the result on the existence of the inverse mapping of graphon controllability Gramian operator  when the system is exactly controllable.

\subsection{Proof of Theorem \ref{theorem: existence of graphon Gramian inverse}}
\begin{proof}
If the graphon system $(\BA; \BB)$ is exactly controllable, then 
$$
   \forall h\in L^{2}[0,1], \quad \exists  c_T > 0, \quad (\BW_T h, h)\geq c_T\|h\|^2.
$$
Let $\BI$ denote the identity operator from $L^2[0,1]$ to $L^2[0,1]$.
Let $\mathbb{M}=\BW_T-\frac12c_T\BI$, then $\BW_T=\frac12c_T(\BI+\frac2{c_T}\mathbb{M})$. By definition, the operator $\mathbb{M}$ is nonnegative and symmetric and hence  $\frac2{c_T}\mathbb{M}$ is nonnegative and symmetric.  
By Proposition \ref{prop: inverse operator exists}, $(\BI+\frac2{c_T}\mathbb{M})$  is one-to-one and onto and the inverse operator is bounded. By a scaling factor $\frac12c_T$, $\BW_T=\frac12c_T(\BI+\frac2{c_T}\mathbb{M})$ is one-to-one and onto and hence the inverse operator $\BW_T^{-1}$ exists. Since the scaling factor $\frac12c_T$ is strictly positive and finite, the inverse operator $\BW_T^{-1}$ is also bounded.   
\end{proof}

\subsection{Proof of Proposition \ref{prop:inverse of the controllability Gramian operator}}
\begin{proof}
  The controllability Gramian is given by 
  \begin{equation}
    \begin{aligned}
      \BW_T & = \int_0^T e^{\FA t} e^{\FA^\TRANS t} dt
            = \int_0^T \left(\BI + \sum_{i=1}^{\infty} (2\FA t)^i \frac{1}{i!} \right) dt \\
           & = T\BI + \sum_{\ell\in I_\lambda} \left(\frac1{2\lambda_\ell}[e^{2\lambda_\ell T}-1]-T \right) \Ff_\ell \Ff_\ell^\TRANS .
    \end{aligned}
  \end{equation}
  Suppose $\Fu=\BW_T\Fx$. We need to find the operator that maps $\Fu$  to $\Fx$. So we set
  $$
  \begin{aligned}
    \Fu=\BW_T\Fx= T\Fx +\sum_{\ell\in I_\lambda} \left(\frac1{2\lambda_\ell}[e^{2\lambda_\ell T}-1]-T \right) \Ff_\ell \Ff_\ell^\TRANS \Fx.
  \end{aligned}
  $$
  Therefore 
  $
  \Fx = \frac{1}T \Fu - \frac{1}T \sum_{\ell\in I_\lambda}\left(\frac1{2\lambda_\ell}[e^{2\lambda_\ell T}-1]-T \right) \Ff_\ell \Ff_\ell^\TRANS \Fx.
  $
  From the definition of $\Fu$, we obtain 
  $\Ff_\ell^\TRANS \Fu %
   = \left(\frac1{2\lambda_\ell}[e^{2\lambda_\ell T}-1]\right)\Ff^\TRANS_\ell \Fx.  $
  Therefore 
  $$
  \Fx = \frac{1}T \Fu - \frac{1}T \sum_{\ell\in I_\lambda}\frac{\left(\frac1{2\lambda_\ell}[e^{2\lambda_\ell T}-1]-T \right)}{\left(\frac1{2\lambda_\ell}[e^{2\lambda_\ell T}-1] \right)}  \Ff_\ell \Ff_\ell^\TRANS \Fu.
  $$
  Equivalently, we obtain the result in \eqref{equ:gramian-inverse}.
\end{proof}

\section{Proofs for  Graphon-LQR}\label{sec:Proofs for  Graphon-LQR}
%
%%%%%%%%%%%%%%

%
 \subsection{Proof of Lemma \ref{lem:Convergnece of Approximated Riccati Solution}}
 \begin{proof}

By Lemma \ref{lem:uniform-converge-Approx} and the definition of the convergence in $C_s([0,T]; \Sigma(L^2[0,1]))$, we obtain that
for any $\Fx \in L^2[0,1]$
 \begin{equation}
    \begin{aligned}
  &\lim_{N\rightarrow \infty} \sup_{t\in[0,T]}\|\BAPt{t} \Fx- \BP_t \Fx\|_2 =0,\\
  &\lim_{N\rightarrow \infty} \sup_{t\in[0,T]}\| \BP_{t} \Fx -\BSPt{t} \Fx \|_2=0.
 \end{aligned}
 \end{equation}
 Since
  	\begin{equation}
        \begin{aligned}
    & \sup_{t\in[0,T]}\|\BAPt{t} \Fx- \BSPt{t} \Fx \|_2  \\
    & \leq \sup_{t\in[0,T]} \|\BAPt{t} \Fx- \BP_{t} \Fx\|_2 + \sup_{t\in[0,T]} \| \BP_{t}\Fx -\BSPt{t} \Fx \|_2, 
      \end{aligned}
    \end{equation}
 we obtain
$
  \lim_{N\rightarrow \infty} \sup_{t\in[0,T]} \|\BAPt{t} \Fx- \BSPt{t} \Fx \|_2 =0.
$  ~
 \end{proof}

\begin{lemma}\label{lem:uniform-bound-for-strong-convergence}
  If a sequence $\{\BT_N \in \mathcal{L}(L^2[0,1])\}$ of bounded linear operators for $L^2[0,1]$ functions  converges strongly to $\BT \in \mathcal{L}(L^2[0,1])$, that is, 
  \begin{equation}
    \forall \Fx \in L^2[0,1],\quad \lim_{N\rightarrow \infty}\|\BT_N \Fx - \BT \Fx\|_2 =0,
  \end{equation}
  then there exists $c>0$ such that
  \begin{equation*}
    \forall N \in \BN, \quad \|\BT_N\|_{\textup{op}} \leq c. 
  \end{equation*}
\end{lemma}
\begin{proof}
Consider any fixed $\Fx \in L^2[0,1]$ and an arbitrarily fixed $\varepsilon>0$.
  The strong convergence of $\{\BT_N \}$ implies  there exist $N_0>0$ such that for $N>N_0$,
     $\|\BT_N\Fx - \BT\Fx\|_2 \leq \varepsilon,$
which implies $\|\BT_N \Fx\|_2 \leq \varepsilon +\|\BT \Fx\|_2$ for $N>N_0$. Let $L = \max\{\|\BT_N \Fx\|_2: 1\leq N \leq N_0\}$. Then $ \|\BT_N \Fx\|_2 \leq \max \{L,  \varepsilon +\|\BT \Fx\|_2\}.$
That is, for any fixed $\Fx \in L^2[0,1]$, $\|\BT_N \Fx\|_2$ is uniformly bounded in $N$.  
Since $\BT_N$ is a linear bounded operator from $L^2[0,1]$ to $L^2[0,1]$, the Uniform Boundedness Principle applies here and hence 
$\|\BT_N\|_{\text{op}}$ is uniformly bounded in $N$.
~
\end{proof}

\subsection{Proof of Theorem \ref{thm: Convergence of the States}}

\begin{proof}
	The closed loop system with the optimal control law is given by
	\begin{multline} \label{equ:optimal-closed-loop-equ}
		 \dot{\Fx}_\Ft^{N*}= \left(\BSA - \BSB \BSB^\TRANS \BSPt{(T-t)} \right) \Fx_\Ft^{N*}, \quad
		 t\in [0,T];
	\end{multline}
	the closed loop system under the  graphon approximate control law is given by 
	\begin{multline} \label{equ:approx-closed-loop-equ}
		 {\DSxt}= \left(\BSA - \BSB \BSB^\TRANS \BAPt{(T-t)} \right)\Sxt,  
		 \quad t\in [0,T].
	\end{multline}
  The initial conditions for \eqref{equ:optimal-closed-loop-equ} and \eqref{equ:approx-closed-loop-equ} are given by $\Fx_0 \in L^2[0,1]$.
	Let $\mathbf{e_t^{N}} := \Fx_\Ft^{N*}- \Sxt$. 
	By (\ref{equ:optimal-closed-loop-equ}) and (\ref{equ:approx-closed-loop-equ}), we obtain
	\begin{equation}\label{equ:xe-differential equation}
		\begin{aligned}
		&\mathbf{\dot{e}_t^{N}}   = \mathbb{F}^N_t\mathbf{e_t^{N}}  + {\Fv}^\FN_\Ft,\\
    & \mathbb{F}^N_t  := \left(\BSA- \BSB \BSB^\TRANS \BSPt{(T-t)} \right),\\
  & \Fv^\FN_\Ft  :=\BSB \BSB^\TRANS \left( \BSPt{(T-t)} - \BAPt{(T-t)}   \right) \Fx_\Ft^{N*}.
		\end{aligned}
	\end{equation}
Since $\mathbf{e_0^{N}}  = 0 \in L^2[0,1]$,  the integral representation of (\ref{equ:xe-differential equation}) is given by 
$  \mathbf{e_t^{N}}  = \int_0^t \Fv^\FN_\tau d\tau + \int_0^t \mathbb{F}^N_\tau \mathbf{e_\tau^{N}}  d\tau.$
Hence we obtain 
\begin{equation}
  \|\mathbf{e_t^{N}}\|_2  \leq \int_0^t \|\Fv^\FN_\tau\|_2 d\tau + \int_0^t \|\mathbb{F}^N_\tau\|_{\text{op}} \|\mathbf{e_\tau^{N}}\|_2  d\tau.
\end{equation}
Applying the Gr\"onwall-Bellman inequality \cite[p.7]{bebernes1968differential}, we obtain 

\begin{equation}\label{equ:gronwall}
\begin{aligned}
   \|\mathbf{e_t^{N}} \|_2  \leq %
             \int_0^t e^{\int_s^t \|\mathbb{F}^N_\tau\|_{\text{op}} d\tau} \|\Fv^\FN_s\|_2ds .
\end{aligned}
\end{equation}

By the convergence of $\BSA, \BSB$ and $\BSP$, we obtain that  the limit $\mathbb{F}:= \BA -\BB \BB^\TRANS \BP_{(\cdot)}$ of the sequence $\{\mathbb{F}^N \}$  exists $\text{in } C_s([0,T]; \Sigma(L^2[0,1]))$, that is, 
\begin{equation}\label{equ:FN-convergence}
  \mathbb{F}^N \rightarrow \mathbb{F} \text{ as } N \rightarrow \infty  \text{ in } C_s([0,T]; \Sigma(L^2[0,1]));
\end{equation}
furthermore,  $\mathbb{F}$ is bounded $\text{in } C_s([0,T]; \Sigma(L^2[0,1]))$, that is, 
\begin{equation}
  \forall \Fx \in L^2[0,1], \quad \sup_{t\in[0,T]}\|\mathbb{F}_t \Fx\|_2 <\infty.
\end{equation}
By the convergence of $\{\mathbb{F}^N \}$ in  $C_s([0,T]; \Sigma(L^2[0,1]))$, we obtain that for every 
$\Fx \in L^2[0,1]$ there exists  $c>0$ such that 
$ \sup_{t\in[0,T]} \|\mathbb{F}^N_t \Fx\|_2 \leq c$
holds for all $N$.  Therefore,% we know
\begin{equation}
  \forall \Fx \in L^2[0,1], \quad \sup_{\BT \in \mathcal{F}} \|\BT \Fx\|_2 <\infty,
\end{equation}
where %
$\mathcal{F}:=\{\mathbb{F}^N_t: t\in [0, T], N \in \{1, \ldots,\infty \}  \}.$
By the Uniform Boundedness Principle \cite[Chapter III, Theorem 14.1]{conway2013course}, % 
there exists $c_{_F}>0$ such that 
\begin{equation}\label{equ:uniform-bound}
  \sup_{\BT \in \mathcal{F}} \|\BT \|_{\text{op}} \leq c_{_F}.
\end{equation}
Furthermore, since $\{\BSB\}$ converges strongly to $\BB$, $\{\|\BSB\|_{\text{op}} \}$ is uniformly bounded in $N$ by Lemma \ref{lem:uniform-bound-for-strong-convergence}. Let $c_{_B}$ denote this uniform bound. 
This, together with  \eqref{equ:gronwall} and \eqref{equ:uniform-bound}, yields
\begin{equation}
\begin{aligned}
   &\|\mathbf{e_t^{N}} \|_2  \leq %
             \int_0^t e^{(t-s)c_{_F}} \left\|\Fv^\FN_s \right\|_2ds \\
             & \leq c^2_{_B} e^{t c_{_F}}  
             \int_0^t \left\|\left( \BSPt{(T-s)} - \BAPt{(T-s)}   \right) \Fx_s^{N*}\right\|_2ds\\
             &\leq   c^2_{_B} e^{t c_{_F}} \int_0^t \Big( \big\|\big( \BSPt{(T-s)} - \BAPt{(T-s)}   \big)\|_{\text{op}} \| \big(\Fx_s^{N*}-\Fx_s^*\big)\big\|_2\\
   &\qquad\qquad +
\big \| \big( \BSPt{(T-s)} - \BAPt{(T-s)}   \big) \Fx_s^{*}\big\|_2 \Big)
  ds,
\end{aligned}
\end{equation}
where $\Fx^*$ denotes the trajectory of system $(\SA; \SB)$ with the optimal control in \eqref{equ:optimal-control} under the initial condition $\Fx^*_0$ which is the limit of the convergent sequence $\{\SxZ\}$.
Note that
\begin{equation}\label{equ:uniform-bound-state}
  \sup_{t\in[0,T]}\|\Fx_\Ft^{N*}\|_2 =  \sup_{t\in[0,T]} \|e^{\int_0^t \mathbb{F}_s^N ds}\Fx_0^{N*}\|_2\leq e^{Tc_{_F}}\|\Fx_0^{N*}\|_2,
\end{equation}
where the initial conditions $\{\Fx_0^{N*} \in L^2[0,1]\}$ are assumed to converge to $\Fx^*_0$. 
Hence
\begin{equation}\label{equ:uniform-converge-state}
\begin{aligned}
  \sup_{t\in[0,T]}\|\Fx_\Ft^{N*}- \Fx_\Ft^*\|_2 &=  \sup_{t\in[0,T]} \|e^{\int_0^t \mathbb{F}_s^N ds}(\Fx_0^{N*}-\Fx_0^*)\|_2\\
  & \leq e^{Tc_{_F}}\|\Fx_0^{N*}-\Fx_0^*\|_2,
\end{aligned}
\end{equation}
that is, $\{\Fx_\Ft^{N*}\}$ converges to $\Fx_\Ft^*$ with respect to $N$ and are uniformly bounded with respect to $t$ over the horizon $[0,T]$. 

Following a similar argument in \eqref{equ:uniform-bound}, $\{\|\BSPt{t} - \BAPt{t}\|_{\text{op}} \}$ is uniformly bounded for all $t\in[0,T]$ and all $N \in \BN$. This, together with the uniform convergence of  $\{\Fx_\Ft^{N*}\}$  to  $\{\Fx_\Ft^{*}\}$ in \eqref{equ:uniform-converge-state} and the result in Lemma \ref{lem:Convergnece of Approximated Riccati Solution}, implies for any $t\in [0,T]$,
\begin{equation}
  \lim_{N\rightarrow \infty}\|\mathbf{e_t^{N}} \|_2 =0,\quad  \text{i.e.} \quad \lim_{N \rightarrow \infty} \|\Fx^{N*}_t - \Sxt\|_2 =0.
\end{equation}

Next we prove the convergence of the cost function. 
\begin{equation}
	\begin{aligned}
		&\big|J(\Fu^{N*})- J(\Su) \big| \\
    & \leq \int_0^T  \Big( \big \|\BSC (\Fx_\Ft^{N*}+\Sxt)\big\|_2 \big\|\Fx_\Ft^{N*}- \Sxt\big\|_2 \\
    & \qquad + \big\|\Fu_\Ft^{N*}+\Sut \big \|_2 \big\|\Fu_\Ft^{N*}-\Sut\big\|_2\Big) dt \\
    &\qquad +\left\| \BSPZ(\Fx_\FT^{N*}+\SxT)\right\|_2\left\| \Fx_\FT^{N*}-\SxT \right\|_2\\
	\end{aligned}
\end{equation}
Since $\Fx_0^{N*} = \SxZ$, following a similar argument in \eqref{equ:uniform-bound-state} we obtain there exists $c_{_{F1}}>0$ such that 
\begin{equation}
     \sup_{t\in[0,T]}\|\Sxt\|_2 \leq e^{Tc_{_{F1}}}\|\Fx_0^{N*}\|_2.
 \end{equation}
The convergence of $\{\Fx_0^{N*}\}$ implies that there exists a uniform bound $c_x>0$ for $\Fx_t^{N*}$ and $\Sxt$ for all $N \in \{1,2,\ldots\}$ and all $t \in[0,T]$. 
  The strong convergence of $\{\BSC\}$ implies that $\{\|\BSC\|_{\text{op}}\}$ is uniformly bounded in $N$ by some constant $c_{_C}>0$ (see Lemma \ref{lem:uniform-bound-for-strong-convergence}). Therefore,
\begin{equation}
\begin{aligned}
   &\lim_{N\rightarrow \infty}\int_0^T  \Big( \big \|\SC (\Fx_\Ft^{N*}+\Sxt)\big\|_2 \big\|\Fx_\Ft^{N*}- \Sxt\big\|_2\Big) dt\\
  &\leq \lim_{N\rightarrow \infty}c_x c_{_C} \int_0^T   \big\|\Fx_\Ft^{N*}- \Sxt\big\|_2dt\\
  & \leq c_x c_{_C}  \int_0^T \lim_{N\rightarrow \infty}   \big\|\Fx_\Ft^{N*}- \Sxt\big\|_2dt =0
\end{aligned}
\end{equation}
Recall
 \begin{equation}
  \begin{aligned}
    &\Fu_\Ft^{N*} = -\BSB \BSPt{(T-t)} \Fx_\Ft^{N*}, 
    &\Sut = -\BSB \BAPt{(T-t)} \Sxt.
  \end{aligned}
\end{equation}
 Note that $\{\|\BSB\|_{\text{op}} \}$ is uniformly bounded in $N$, and $\{\|\BSPt{t}\|_{\text{op}}\}$ and $\{\|\BAPt{t}\|_{\text{op}}\}$ are uniformly bounded in $N$ and in $t$. These, together with the fact that $\Fx_\Ft^{N*}$ and $\Sxt$ are uniformly bounded in $N$ and $t$, imply that there exists a uniform bound $c_u>0$ for 
 $\{\Fu_\Ft^{N*} \}$ and $\{\Sut\}$ in $N$ and in $t$. Hence
\begin{equation}
\begin{aligned}
   \int_0^T\Big( &\big\|\Fu_\Ft^{N*}+\Sut \big \|_2 \big\|\Fu_\Ft^{N*}-\Sut\big\|_2\Big) dt\\
   &\leq c_u \int_0^T\big\|\Fu_\Ft^{N*}-\Sut\big\|_2 dt \\
   &\leq c_u c_{_B}\int_0^T\Big(\big\| \big(\BSPt{(T-t)} -\BAPt{(T-t)}\big) \Fx_\Ft^{N*}\big\|_2 \\
   &+ \big\|\BAPt{(T-t)} \big( \Fx_\Ft^{N*}-\Sxt \big)\big\|_2\Big)dt.
\end{aligned}
\end{equation}
Hence, 
\begin{equation}
  \lim_{N\rightarrow \infty}\int_0^T\Big( \big\|\Fu_\Ft^{N*}+\Sut \big \|_2 \big\|\Fu_\Ft^{N*}-\Sut\big\|_2\Big) dt =0.
\end{equation}
A similar argument yields 
\begin{equation}
  \begin{aligned}
  &  \lim_{N\rightarrow \infty} \big\|\BSPZ(\Fx_\FT^{N*}+\SxT)\big\|_2 \big\|\Fx_\FT^{N*}-\SxT \big\|_2 =0.
\end{aligned}
\end{equation}
  Therefore, we have
	 	$\lim_{N\rightarrow \infty} \left|J(\Fu^{N*})- J(\Su)\right|=0.$
	  ~
\end{proof}	
%%%%%%%%%%%%%%%%%%
%
\bibliographystyle{IEEEtran}
\bibliography{IEEEabrv,mybib}

\vspace{-1cm}
\begin{IEEEbiography}[{\includegraphics[draft=false, width=1in,height=1.25in,clip,keepaspectratio]{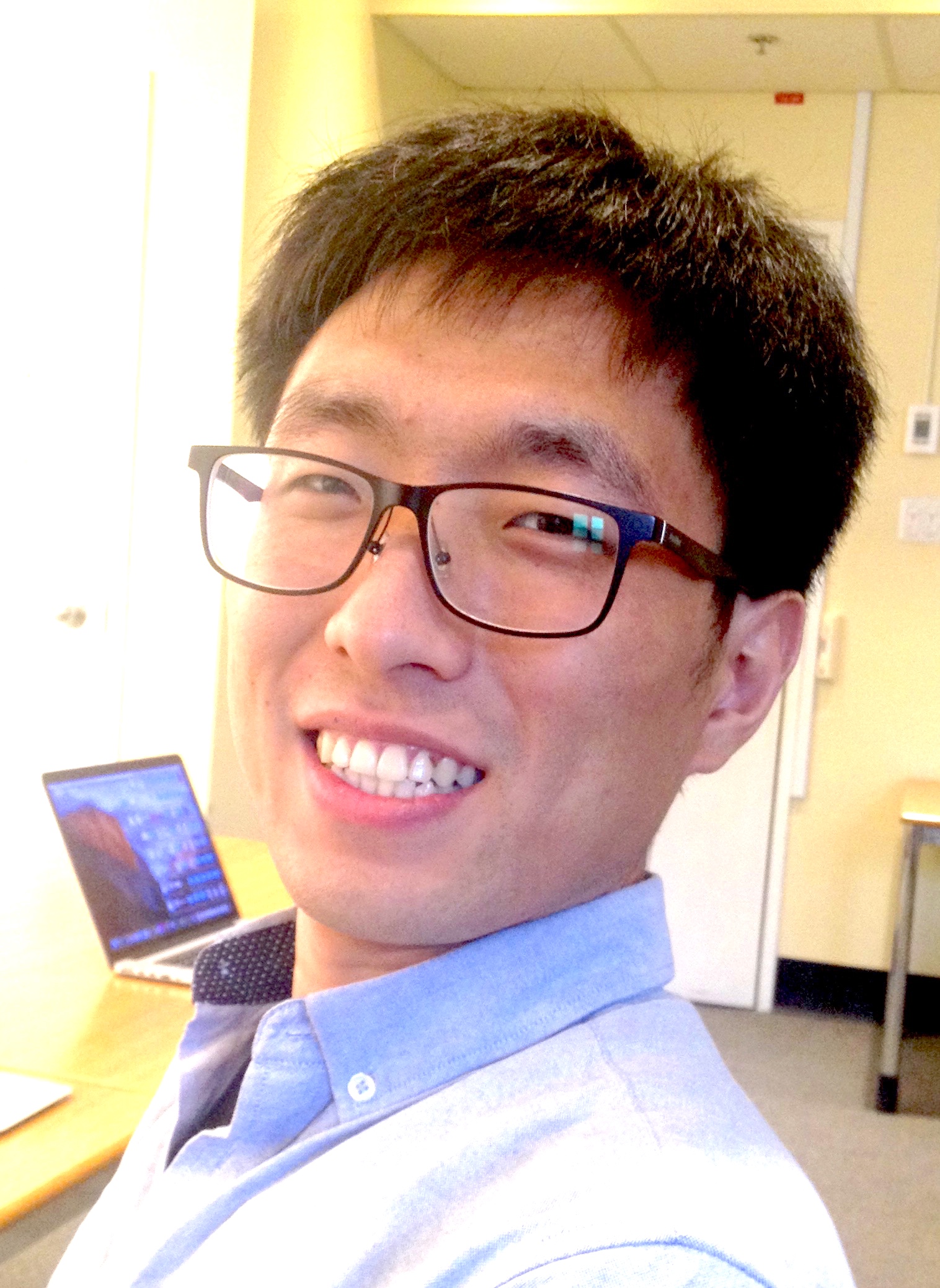}}]{Shuang Gao} (M'14) received the B.E. degree in automation  and M.S. in control science and engineering, from Harbin Institute of Technology, Harbin, China, in 2011 and 2013. He received the Ph.D. degree in electrical engineering from McGill University, Montreal, QC, Canada, in February 2019, under the supervision of Prof. Peter. E. Caines. He is a member of McGill Centre for Intelligent Machines and Groupe d'\'Etudes et de Recherche en Analyse des D\'ecisions. His research interest includes control of network systems, optimization on networks, decentralized control, network modeling, mean field games.
\end{IEEEbiography}
\vspace{-1cm}

\begin{IEEEbiography}[{\includegraphics[draft=false,width=1in,height=1.25in,clip,keepaspectratio]{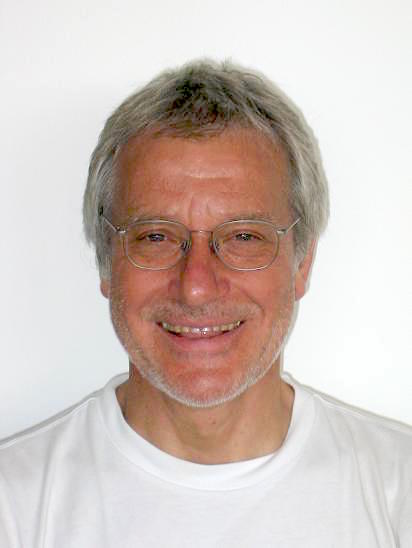}}]{Peter E. Caines} (LF'11)
received the BA in mathematics from Oxford University in 1967 and the PhD in systems and control theory in 1970 from Imperial College, University of London, under the supervision of David Q. Mayne, FRS. After periods as a postdoctoral researcher and faculty member at UMIST, Stanford, UC Berkeley, Toronto and Harvard, he joined McGill University, Montreal, in 1980, where he is Distinguished James McGill Professor and Macdonald Chair in the Department of Electrical and Computer Engineering. In 2000 the adaptive control paper he coauthored with G. C. Goodwin and P. J. Ramadge (IEEE Transactions on Automatic Control, 1980) was recognized by the IEEE Control Systems Society as one of the 25 seminal control theory papers of the 20th century. He is a Life Fellow of the IEEE, and a Fellow of SIAM, the Institute of Mathematics and its Applications (UK) and the Canadian Institute for Advanced Research and is a member of Professional Engineers Ontario. He was elected to the Royal Society of Canada in 2003. In 2009 he received the IEEE Control Systems Society Bode Lecture Prize and in 2012 a Queen Elizabeth II Diamond Jubilee Medal. Peter Caines is the author of Linear Stochastic Systems, John Wiley, 1988, republished as a SIAM Classic in 2018, and is a Senior Editor of Nonlinear Analysis-Hybrid Systems; his research interests include stochastic, mean field game, decentralized and hybrid systems theory,  together with their applications in a range of fields.
\end{IEEEbiography}
%%%%%%%%%%%%%%%%%%%%%%%%%%%%%%%%%%%%%%%%%%%%%%%%%%%%%%
\end{document}